\documentclass[12pt]{amsart}
\usepackage[dvips]{graphics}
\usepackage[latin1]{inputenc}
\usepackage{amssymb, amsmath, amscd, mathrsfs,marvosym, psfrag}
\input xy
\xyoption{all}
\DeclareMathAlphabet{\mathpzc}{OT1}{pzc}{m}{it}

\newtheorem{theorem}{Theorem}[section]

\newtheorem{proposition}[theorem]{Proposition}
\newtheorem{corollary}[theorem]{Corollary}
\newtheorem{conjecture}[theorem]{Conjecture}
\newtheorem{lemma}[theorem]{Lemma}

\theoremstyle{definition}
\newtheorem{definition}[theorem]{Definition}

\theoremstyle{remark}
\newtheorem{remark}[theorem]{Remark}
\newtheorem{Rem}[theorem]{Remark}

\def\varle{\leqslant}

\newcommand{\CA}{{\mathcal A}}

\newcommand{\CI}{{\mathcal I}}
\newcommand{\CJ}{{\mathcal J}}
\newcommand{\CK}{{\mathcal K}}

\newcommand{\CO}{{\mathcal O}}

\newcommand{\CW}{{\mathcal W}}

\newcommand{\CZ}{{\mathcal Z}}

\newcommand{\fh}{{{\mathfrak h}}}

\newcommand{\fg}{{{\mathfrak g}}} 
\newcommand{\fb}{{{\mathfrak b}}}

\newcommand{\fn}{{{\mathfrak n}}}

\newcommand{\fhd}{\fh^\star}

\newcommand{\hCW}{{\widehat\CW}}

\newcommand{\hPi}{{\widehat\Pi}}
\newcommand{\hfh}{{\widehat\fh}}
\newcommand{\hfg}{{\widehat\fg}}
\newcommand{\hfb}{{\widehat\fb}}
\newcommand{\hfn}{{\widehat\fn}}

\newcommand{\hR}{{\widehat R}}

\newcommand{\hfhd}{{\widehat{\fh}}^\star}

\newcommand{\tfg}{\widetilde\fg}

\newcommand{\DC}{{\mathbb C}}

\newcommand{\DZ}{{\mathbb Z}}

\newcommand{\DN}{{\mathbb N}}

\newcommand{\End}{{\operatorname{End}}}
\newcommand{\Ext}{{\operatorname{Ext}}}
\newcommand{\Hom}{{\operatorname{Hom}}}

\newcommand{\id}{{\operatorname{id}}}

\newcommand{\catmod}{{\operatorname{-mod}}}

\newcommand{\inj}{{\hookrightarrow}}

\newcommand{\dual}{{\mathsf D}}

\newcommand{\ol}{\overline}

\newcommand{\Quot}{{\operatorname{Quot\,}}}

\newcommand{\sur}{\mbox{$\to\!\!\!\!\!\to$}}

\newcommand{\GL}{{\operatorname{GL}}}

\newcommand{\comment}[1]{}

\newcommand{\cprime}{$'$}
\newcommand{\Vcr}{V^{\crit}(\fing)}
\newcommand{\gr}{\mathrm{gr}}

\newcommand{\RZm}{{\mathcal{Z}}^{-}}
\newcommand{\Vcri}{V^{\crit}(\fing)}
\newcommand{\vac}{\mathbf{1}}

\renewcommand{\*}{{\otimes}}
\newcommand{\BRST}[2]{H^{\frac{\infty}{2}+#1}(\mathfrak{n}_-[t,t^{-1}],#2\*
\mathbb{C}_{\widehat{\Psi}})}
\newcommand{\semiwedge}[1]{\bigwedge\nolimits^{\frac{\infty}{2}+#1}}
\newcommand{\lam}{\lambda}
\newcommand{\bra}{{\langle}}
\newcommand{\ket}{{\rangle}}

\newcommand{\fing}{\mathfrak{g}}
\newcommand{\finn}{\mathfrak{n}}

\newcommand{\affg}{\widehat{\mathfrak{g}}}
\newcommand{\affh}{\widehat{\mathfrak{h}}}

\newcommand{\crit}{\mathrm{crit}}

\newcommand{\ra}{\rightarrow}
\newcommand{\rank}{\mathrm{rank}}
\DeclareMathOperator{\cha}{\mathrm{ch}}

\newcommand{\inv}{{}^{-1}}
\newcommand{\C}{\mathbb{C}}
\newcommand{\Z}{\mathbb{Z}}
\newcommand{\Cl}{\mathcal{C}l}
\newcommand{\+}{\mathop{\oplus}}
\newcommand{\imag}{\mathrm{im}}
\newcommand{\real}{\mathrm{re}}

\newcommand{\rCO}{{\ol\CO}}

\newcommand{\rDelta}{{\ol\Delta}}

\begin{document}

\pagenumbering{arabic}
\title[On restricted Verma modules]{On the restricted Verma modules  at the
  critical level}  
\author[]{Tomoyuki Arakawa, Peter Fiebig}
\thanks{T.A. is 
partially  supported 
by the JSPS Grant-in-Aid  for Scientific Research (B)
No.\ 20340007. P.F. is supported by a grant of the Landesstiftung Baden--W\"urttemberg }

\begin{abstract} We study the restricted Verma modules of an affine
  Kac--Moody algebra at the critical level with a special emphasis on
  their Jordan--H\"older multiplicities. Feigin and Frenkel
  conjectured a formula for these multiplicities that involves
  the periodic Kazhdan--Lusztig polynomials. We prove this conjecture
  for all subgeneric blocks and for the case of anti-dominant simple
  subquotients.
\end{abstract}

\maketitle

\section{Introduction}

The representation theory of affine Kac--Moody algebras at the
critical level is one of the essential ingredients in the approach towards the geometric
Langlands conjectures proposed by  Beilinson and Drinfeld (cf.~
\cite{BeiDri96}). In particular, the correspondence between the center
of the (completed) universal enveloping algebra of an affine
Kac--Moody algebra at the critical level and the geometry of the space
of opers associated with  the Langlands dual datum (cf.~\cite{FeiFre92}) is one of the main tools used in the construction of a 
part of the Langlands correspondence in \cite{BeiDri96}. 

\subsection{The local geometric Langlands conjectures}
In \cite{FreGai06} Frenkel and Gaitsgory formulated the {\em local
  geometric Langlands conjectures}, which 
relate the critical level representation theory of an affine Kac--Moody
algebra to the geometry of an affine flag manifold. In a
series of subsequent papers, the authors proved parts of these conjectures. In particular, in the paper \cite{FreGai07} a derived
equivalence between a certain category of $D$-modules on the affine
flag variety and a derived version of the affine category $\CO$ at the
critical level was constructed using a localization
functor of Beilinson--Bernstein type. It seems, however, hard to
control the action of this equivalence on the respective hearts of the
triangulated categories, and hence it is not yet possible to
deduce information on the simple critical characters of the category $\CO$
from the Frenkel--Gaitsgory result.

\subsection{The Andersen--Jantzen--Soergel approach}
In this paper we study the critical representation theory by a very
different method that is inspired by results of  Jantzen
(cf.~\cite{Jan79}) and Soergel (cf.~\cite{Soe90}) in the case of finite dimensional complex Lie
algebras and of  Andersen, Jantzen and Soergel in the case of modular Lie
algebras and quantum groups (cf.~\cite{AJS94}). The analogous results for the
non-critical blocks of $\CO$ for a symmetrizable Kac--Moody algebra
can be found in \cite{Fie06}. 

The main idea of
this approach is to first describe the {\em generic} and {\em
  subgeneric} blocks of the respective representation theory in as
much detail as possible, and then to  deform an arbitrary block
in such a way that it can be viewed as an intersection of generic and
subgeneric blocks. This intersection procedure should then be
described using only some underlying  combinatorial datum (like for
example the associated integral Weyl group). 

As a result we hope to be able to
construct equivalences between various categories or links to
categories defined in topological terms in the framework of the
geometric Langlands program, and to deduce information on
the respective simple characters. More specifically, we hope to find a
correspondence between certain intersection cohomology sheaves on the
Langlands dual affine flag variety and critical representations.

This paper provides the first steps in the approach described above. Its main result is the calculation of the simple
characters in the subgeneric critical blocks. In order to explain our
result, let us consider the respective categories in some  more detail.

\subsection{Critical representations of affine Kac--Moody algebras}
Let us denote by $\hfg$ an affine Kac--Moody algebra and by
$\hfh\subset\hfg$ its Cartan subalgebra (for the specialists it should
be noted here that we consider the central extension of a loop
algebra together with the grading operator). The one-dimensional
center of $\hfg$ acts semisimply on each module in the category $\CO$. Accordingly, each block of the category $\CO$,
i.e.~each of its indecomposable direct summands, 
determines a central character. There is one such character, called
the {\em critical character}, which is distinguished in more than one
respect. 

In this paper we focus on the following feature of the
critical blocks. Let $\hfhd$ be the dual space of the Cartan
subalgebra and denote by $\delta\in\hfhd$ the smallest positive
imaginary root. Then the corresponding simple highest weight module
$L(\delta)$ is of dimension one, and the tensor product
$\cdot\otimes_\DC L(\delta)$ defines a {\em shift functor} $T$ on
$\CO$ that is an equivalence. Now the critical blocks are exactly
those that are preserved by the functor $T$. This allows us to
consider, for each critical block, the 
corresponding graded center $\CA$ (see Section \ref{sec-gradedcenter}).

The graded center  is huge and  intimately related  to the center of the
(completed) universal enveloping algebra at the critical level, which
was determined in the fundamental work of Feigin and Frenkel
 (cf.~\cite{FeiFre92}) (conjecturally, $\CA$ is a quotient of the
 latter center). We
use the results of Feigin--Frenkel to describe the action of $\CA$ on
the Verma modules contained in the critical blocks.

\subsection{The restricted Verma modules} 
Let  $\Delta(\lambda)$ be the Verma module
with highest weight $\lambda\in\hfhd$. We define the {\em restricted
  Verma module} $\rDelta(\lambda)$ with highest weight $\lambda$ as the quotient of $\Delta(\lambda)$ by
the ideal of $\CA$ generated by the homogeneous
constituents  $\CA_n$
with $n\ne 0$.  For our approach, the restricted Verma
modules, not the ordinary ones, should be considered as the ``standard
objects'' in the critical blocks. 

We denote the irreducible quotient of $\Delta(\lambda)$ by
$L(\lambda)$. Results of Frenkel and
Feigin--Frenkel yield the characters of the restricted Verma modules,
and the knowledge of the character of $L(\lambda)$ for all $\lambda$
is equivalent to the knowledge of the Jordan--Hölder multiplicities
$[\rDelta(\lambda):L(\mu)]$ for all pairs $\lambda,\mu$ of critical
weights. 

\subsection{The Feigin--Frenkel conjecture}
Let us choose a critical indecomposable block $\CO_\Lambda$ of $\CO$
and let us identify the index $\Lambda$ with the subset of highest
weights of the simple modules in $\CO_\Lambda$. Since $\Lambda$ is
critical we have $\lambda+\delta\in\Lambda$ if and only if
$\lambda\in\Lambda$. 

Let $\hCW$ be the affine Weyl group associated with our data, and
$\CW\subset\hCW$ the finite Weyl group. The {\em integral Weyl group}
$\hCW(\Lambda)$ corresponding to $\Lambda$ is generated by the
reflections with respect to those real roots that satisfy a certain
integrality condition with respect to $\Lambda$. In the critical case,
$\hCW(\Lambda)$ is the affinization of the corresponding {\em finite  integral Weyl group} $\CW(\Lambda)\subset\CW$.

We define $\lambda\in\Lambda$ to be {\em dominant} resp.~ {\em
  anti-dominant} if it is
dominant resp.~ anti-dominant with respect to the action of
$\CW(\Lambda)$, i.e.~if it is the highest resp.~ smallest element
in its $\CW(\Lambda)$-orbit. We say that $\lambda$ is {\em regular} if it is
regular with respect to $\CW(\Lambda)$ (note that here we only refer to the finite integral Weyl group).

As mentioned before,  $\hCW(\Lambda)$ is the affinization of
$\CW(\Lambda)$. In \cite{Lus80} Lusztig associated with a pair
$w,x\in\hCW(\Lambda)$ the {\em periodic polynomial}
$p_{x,w}\in\DZ[v]$ (in Lusztig's paper these polynomials were indexed not by affine Weyl group elements, but by alcoves, see Section \ref{subs-FFconj} for more details). The {\em Feigin--Frenkel conjecture} is the
following.

\begin{conjecture} Let $\Lambda$ be a critical equivalence class.
\begin{enumerate}
\item {\em The restricted linkage principle:} For
  $\lambda,\mu\in\Lambda$ we have $[\rDelta(\lambda):L(\mu)]=0$
  unless $\lambda$ and $\mu$ are contained in the same
  $\hCW(\Lambda)$-orbit.
\item {\em The restricted Verma multiplicities:} Suppose that
  $\lambda\in\Lambda$ is regular and dominant. Under some further regularity
  conditions on $\Lambda$ (cf.~Section \ref{conj-FFC}), we have
$$
[\rDelta(w.\lambda):L(x.\lambda)]=p_{w,x}(1)
$$
for all $w,x\in\hCW(\Lambda)$.
\end{enumerate}
\end{conjecture}

The Feigin--Frenkel conjecture fits very well into a broader picture
that was anticipated by Lusztig in his ICM address in 1990 in Kyoto
(cf.~\cite{Lus91}). There, Lusztig conjecturally linked the
representation theory of modular Lie algebras, of quantum groups and
of critical level representations of an affine Kac--Moody algebra to 
the topology of semi-infinite flag manifolds.

In \cite{AF2} we use the results of the present article in order to prove part (1) of the Feigin-Frenkel conjecture, the restricted linkage principle.

\subsection{Our main result}
Let $\fh\subset\hfh$ be the finite part of the Cartan subalgebra and let us denote by $\lambda\mapsto \ol\lambda$ the corresponding restriction map from $\hfhd$ to $\fhd$.
We call a critical class $\Lambda$ {\em subgeneric}, if its image $\ol\Lambda$ in $\fhd$ contains precisely two elements. Then we can define, following
\cite{AJS94}, a bijection $\alpha\uparrow\cdot
\colon\Lambda\to\Lambda$. Here, $\alpha$ denotes the unique positive
finite root with $\ol\Lambda=\{\ol\lambda,s_\alpha.\ol\lambda\}$. 
Here is our result:

\begin{theorem}
 \begin{enumerate}
\item If $\nu\in\Lambda$ is 
anti-dominant, then we have for all
  $\gamma\in\Lambda$  
$$
[\rDelta(\gamma):L(\nu)]=
\begin{cases} 1 & \text{ if
    $\gamma\in\CW(\Lambda).\nu$},\\
0 & \text{ otherwise}.
\end{cases}
$$
\item If $\Lambda$ is subgeneric, then we have for all 
$\gamma,\nu\in\Lambda$
$$
[\rDelta(\gamma):L(\nu)] = 
\begin{cases}
1 &\text{ if $\gamma\in\{\nu,\alpha\uparrow\nu\}$},\\
0 &\text{ otherwise}.
\end{cases}
$$
\end{enumerate}
\end{theorem}
The above theorem confirms the Feigin--Frenkel conjecture in the
respective cases. In \cite{AF2} we use it in order to describe the structure of the  {\em restricted} category $\rCO_\Lambda$ completely for subgeneric $\Lambda$. As for generic $\Lambda$ the structure of $\rCO_\Lambda$ is easy to determine  using the results of Feigin and Frenkel (see also \cite{Hay88,Ku89,Mal90,Mat96}, we completed the first part of the Andersen-Jantzen-Soergel approach towards the description of the category $\CO$ at the critical level. 

\subsection{Acknowledgments} Both authors wish to thank the Emmy
Noether Center in Erlangen, where parts of the research for this paper
were done, for its hospitality and
support. The second author wishes to thank Nara Women's University for its hospitality and the Landesstiftung Baden--W\"urttemberg for supporting the project.

\section{Affine Kac--Moody algebras}

In this section we recall  the fundamentals of the theory of affine Kac--Moody algebras. The main references are the
textbooks \cite{Kac90} and \cite{MP95}. Our basic data is a finite
dimensional simple complex Lie algebra $\fg$. We denote by
$k\colon \fg\times \fg\to \DC$ its Killing form. 

\subsection{The construction of $\hfg$}
From $\fg$ we construct the  (untwisted) affine Kac--Moody algebra
$\hfg$ as follows. We first consider the
loop algebra $\fg\otimes_\DC \DC[t,t^{-1}]$ for  which the commutator is
the $\DC[t,t^{-1}]$-linear extension  of the commutator of $\fg$. That
means that we have $[x\otimes t^n,y\otimes t^m]=[x,y]\otimes t^{m+n}$ for
$x,y\in\fg$ and $m,n\in\DZ$. The loop algebra has an up to isomorphism
unique non-split central extension $\tfg$ of rank one. As a
vector space we have 
$\tfg=\fg\otimes_{\DC}\DC[t,t^{-1}]\oplus \DC K$, and the Lie bracket is given by
\begin{align*}
[K,\tfg] &= 0, \\
[x\otimes t^n,y\otimes t^m] &=[x,y]\otimes t^{m+n}+
n\delta_{m,-n}k(x,y) K
\end{align*} 
for $x,y\in\fg$, $n,m\in\DZ$ (here $\delta_{a,b}$ denotes the
Kronecker delta). In the last step of the construction we add the
outer derivation operator $[D,\cdot]=t\frac{\partial}{\partial t}$ and get the affine Kac--Moody algebra $\hfg:=\tfg\oplus \DC D$
with the Lie bracket 
\begin{align*}
[K,\hfg] & = 0, \\
[D,x\otimes t^n] &  = n x\otimes t^n, \\
[x\otimes t^n, y\otimes t^m] & = [x,y]\otimes t^{m+n}+
n\delta_{m,-n}k(x,y) K
\end{align*}
for $x,y\in \fg$, $n,m\in \DZ$.

Let us fix a Borel subalgebra $\fb\subset \fg$ and a Cartan subalgebra
$\fh\subset \fb$. Then the corresponding Borel subalgebra $\hfb$ of $\hfg$ is given by
$$
\hfb  := (\fg \otimes_\DC t\DC[t]+\fb\otimes_\DC \DC[t])\oplus \DC
K\oplus\DC D 
$$
and the Cartan subalgebra $\hfh\subset \hfb$ is given by 
$$
\hfh  := \fh\oplus\DC K\oplus \DC D.
$$

\subsection{Affine roots} We denote by $V^\star$ the dual of a vector space $V$ and we write
$\langle\cdot,\cdot\rangle\colon V^\star\times V\to \DC$ for the
canonical pairing. 
Let $R \subset \fhd$ be the set of roots of $\fg$ with respect to $\fh$. The projection $\hfh\to \fh$ along the decomposition $\hfh=\fh\oplus \DC K\oplus \DC D$ allows us to embed $\fhd$ inside $\hfhd$. In particular, we can view any $\alpha\in R$ as an element in $\hfhd$. Let us define  $\delta,\Lambda_0\in \hfhd$ by 
\begin{align*}
\langle\delta,\fh\oplus \DC K\rangle & = \{0\}, \\ 
\langle\delta,D\rangle & = 1, \\
\langle\Lambda_0,\fh\oplus \DC D\rangle & = \{0\},\\
\langle\Lambda_0,K\rangle & = 1.
\end{align*}
Then we have $\hfhd=\fhd\oplus\DC\Lambda_0\oplus\DC\delta$. The set $\hR\subset \hfhd$ of roots of $\hfg$ with respect to $\hfh$ is 
$$
\hR=\{\alpha+n\delta \mid \alpha\in R,n\in\DZ\}\cup\{n\delta\mid n\in \DZ, n\ne 0\}.
$$
For $\alpha\in R$ let us denote by $\fg_\alpha\subset\fg$ the
corresponding root space. The root spaces of $\hfg$ with respect to $\hfh$ are 
\begin{align*}
\hfg_{\alpha+n\delta} &= \fg_\alpha\otimes t^n\quad\text{for
  $\alpha\in R$, $n\in\DZ$},\\
\hfg_{n\delta} &= \fh\otimes t^n\quad\text{for $n\in\DZ$, $n\ne 0$}.
\end{align*}
The subsets 
\begin{align*}
\hR^{\real} &:=\{\alpha+n\delta \mid \alpha\in R, n\in\DZ\},\\
\hR^{\imag} &:=  \{n\delta\mid n\in \DZ, n\ne 0\}
\end{align*}
are called the sets of {\em real} roots and of {\em imaginary} roots,
resp.

Let $ R^+\subset  R$ be the set of roots of $\fb$ with respect
to $\fh$. Then the set $\hR^+$ of roots of $\hfb$ with respect to $\hfh$ is 
$$
\hR^{+}=\{\alpha+n\delta\mid \alpha\in R, n\ge 1\}\cup
 R^{+}\cup \{n\delta\mid n\ge 1\}.
$$
Let $\Pi\subset R^+$ be the set of simple roots and denote by
$\gamma\in R^+$ the highest root. Then the set of  simple affine roots is 
$$
\hPi=\Pi\cup\{-\gamma+\delta\}\subset \hR^+.
$$

For each real root $\alpha\in \hR^{\real}$ the corresponding root space
$\hfg_{\alpha}$ is one-dimensional, and so is the commutator
$[\hfg_{\alpha},\hfg_{-\alpha}]\subset \hfh$. The {\em (affine) coroot
  $\alpha^\vee$} associated with $\alpha$ is the unique element in
$[\hfg_{\alpha},\hfg_{-\alpha}]$ on which $\alpha$ takes the value
$2$. Note that $\alpha^\vee$ is contained in $\fh\oplus\DC K$, so we
have $\langle\delta,\alpha^\vee\rangle=0$.

\subsection{The Weyl groups}  
For $\alpha\in \hR^{\real}$ we define the reflection $s_\alpha\colon \hfhd\to \hfhd$ by $s_\alpha(\lambda):=\lambda-\langle\lambda,\alpha^\vee\rangle\alpha$. 
We denote by $\hCW\subset \GL(\hfhd)$ the affine Weyl group, i.e.~the
subgroup generated by the reflections $s_\alpha$ for
$\alpha\in \hR^+$. The subgroup $\CW\subset\hCW$ generated by the
reflections $s_\alpha$ with $\alpha\in R$  leaves the subset
$\fhd\subset \hfhd$ stable and can be identified with the Weyl
group of $\fg$.

Let $\rho\in\hfhd$ be an element with the property
$\langle \rho,\alpha^\vee\rangle=1$ for each simple affine root $\alpha$. Note that
$\rho$ is only defined up to the addition of a multiple of $\delta$
(the span of the affine coroots is $\fh\oplus\DC K$, and the simple
coroots form a basis in this space). Yet all constructions in the following that
use $\rho$ do not depend on this choice. So let us fix such an
element $\rho$ once and for all.

 The {\em dot-action} $\hCW\times\hfhd\to\hfhd$,
  $(w,\lambda)\mapsto w.\lambda$, of the affine Weyl group on
  $\hfhd$ is obtained by shifting the linear action in such a way that
  $-\rho$ becomes a fixed point, i.e.~ it is given by
$$
w.\lambda:=w(\lambda+\rho)-\rho
$$
for $w\in\hCW$ and $\lambda\in \hfhd$. Note that since
$\langle\delta,\alpha^\vee\rangle=0$ we have 
$s_\alpha(\delta)=\delta$ for all
$\alpha\in \hR^{\real}$. Hence 
$w(\delta)=\delta$ for all $w\in\hCW$ (so the dot-action is
independent of the choice of $\rho$, as we claimed above).  

\subsection{The invariant bilinear form}
Denote by $(\cdot,\cdot)\colon \hfg\times\hfg\to\DC$  the form given by 
\begin{align*}
(x\otimes t^n, y\otimes t^m) &= \delta_{n,-m} k(x,y), \\
(K, \fg\otimes_\DC \DC[t,t^{-1}]\oplus \DC K) &= \{0\}, \\
(D, \fg\otimes_\DC \DC[t,t^{-1}]\oplus \DC D) &= \{0\}, \\
(K,D) &= 1
\end{align*}
for $x,y\in \fg$, $m,n\in\DZ$. It is  non-degenerate, symmetric and {\em invariant}, i.e.~
it satisfies   $([x,y],z)=(x,[y,z])$ for all
$x,y,z\in\hfg$. Moreover, it induces a non-degenerate bilinear form on the Cartan
subalgebra $\hfh$ and hence yields an isomorphism $\hfh\stackrel{\sim}\to\hfhd$, which is the direct sum of the isomorphism $\fh\to\fhd$ given by the Killing form $k$ and the isomorphism $\DC K\oplus \DC D\to \DC\Lambda_0\oplus \DC \delta$ that maps $K$ to $\delta$ and $D$ to $\Lambda_0$.  We get an induced symmetric
non-degenerate form on the  dual $\hfhd$ that is given explicitly by
\begin{align*}
(\alpha,\beta) &= k(\alpha,\beta), \\
(\Lambda_0, \fhd\oplus \DC\Lambda_0) &= \{0\}, \\
(\delta, \fhd\oplus \DC\delta) &= \{0\}, \\
(\Lambda_0,\delta) &= 1
\end{align*}
for $\alpha,\beta\in \fhd$ (here we denote by $k$ also the form on $\fhd$ that is induced by the Killing form).
It is
invariant under the linear action of the affine Weyl group, i.e.~ we have
$$
(w(\lambda),w(\mu))=(\lambda,\mu) 
$$
for $w\in\hCW$, $\lambda,\mu\in\hfhd$.

\section{The category $\CO$ for an affine Kac--Moody algebra}

Having recalled the fundamental structural results for an affine
Kac--Moody algebra  we now turn to its representation theory. We
restrict ourselves to representations in the affine category $\CO$.  

\subsection{The category $\CO$}
Let $M$ be a $\hfg$-module. Its {\em weight space} corresponding to
$\lambda\in\hfhd$ is  
$$
M_\lambda:=\{m\in M\mid h.m=\langle\lambda,h\rangle m\text{ for all $h\in \hfh$}\}.
$$
Any non-zero element $m\in M_\lambda$ is said to be {\em of weight $\lambda$}.
We say that  $M$ is a {\em weight module} if
$M=\bigoplus_{\lambda\in\hfhd} M_\lambda$. 
 We say that  $M$ is  {\em locally $\hfb$-finite} if all finitely generated $\hfb$-submodules of $M$ are finite dimensional.
The {\em affine category $\CO$} is defined  as the full subcategory of the category of $\hfg$-modules that consists of all locally $\hfb$-finite weight modules.

\subsection{Highest weight modules}
Our choice of positive roots defines a partial order on $\hfhd$: we set $\nu\varle\nu^\prime$ if and only if
$\nu^\prime-\nu\in\DZ_{\ge 0}\hR^+$. 
A {\em highest weight module} of highest weight $\lambda\in \hfhd$ is
a $\hfg$-module $M$ that contains  a generator $v\ne 0$  of weight
$\lambda$ such that 
$\hfg_\alpha v=0$ for all $\alpha\in\hR^+$. Then $\lambda$ is indeed
the highest weight of $M$, i.e.~ $M_\mu\ne 0$ implies
$\mu\varle\lambda$.  Each highest weight module is contained in $\CO$.

For $\lambda\in\hfhd$ denote by $\DC_\lambda$ the one-dimensional
$\hfh$-module corresponding to $\lambda$. We extend the $\hfh$-action
to a $\hfb$-action using the homomorphism $\hfb\to\hfh$ of Lie
algebras that is left inverse to the inclusion $\hfh\subset \hfb$. That
means that $\hfg_\alpha$ acts trivially on $\DC_\lambda$ for all
$\alpha\in\hR^+$. The induced module
$$
\Delta(\lambda):=U(\hfg)\otimes_{U(\hfb)}\DC_\lambda
$$
is called the {\em Verma module} corresponding to $\lambda$. It
contains a unique simple quotient $L(\lambda)$, and both
$\Delta(\lambda)$ and $L(\lambda)$ are highest
weight modules of highest weight $\lambda$.

Moreover, the modules $L(\lambda)$ for $\lambda\in\hfhd$ form a
full set of representatives of the simple isomorphism classes of
$\CO$, i.e.~each simple object in $\CO$ is isomorphic to
$L(\lambda)$ for a unique $\lambda\in\hfhd$.

\subsection{Characters}
Let $\DZ[\hfhd]=\bigoplus_{\lambda\in \hfhd}\DZ e^{\lambda}$ be the group algebra of the additive group $\hfhd$. Let $\widehat{\DZ[\hfhd]}\subset \prod_{\lambda\in\hfhd} \DZ e^{\lambda}$ be the subgroup of elements $(c_\lambda)$ that have the property that there exists a finite set $\{\mu_1,\dots, \mu_n\}\subset\hfhd$ such that $c_\lambda\ne 0$ implies $\lambda\le \mu_i$ for at least one $i$.  

Let $\CO^f\subset \CO$ be the full subcategory of modules $M$ that have
the property that the weight spaces $M_\lambda$ are finite dimensional
and such that there exist  $\mu_1,\dots, \mu_n\in \hfhd$ such  that
$M_\lambda\ne 0$ implies $\lambda\le \mu_i$ for at least one $i$. For
each object $M$ of $\CO^f$ we can then define its {\em character} as
$$
\cha M:=\sum_{\lambda\in \hfhd} (\dim_\DC M_\lambda) e^\lambda\in \widehat{\DZ[\hfhd]}.
$$

The character of a Verma module is easy to calculate. For each $\lambda\in\hfhd$ we have
$$
\cha \Delta(\lambda)=e^{\lambda} \prod_{\alpha\in \hR^+}(1+e^{-\alpha}+e^{-2\alpha}+\dots)^{\dim\fg_\alpha}.
$$
(The above product is well-defined in $\widehat{\DZ[\hfhd]}$.) If $\lambda\in\hfhd$ is non-critical (cf.~ Section \ref{sec-crithyp}), the character of $L(\lambda)$ is known (cf.~\cite{KT00,Fie06}). The principal aim of our research project is to calculate $\cha L(\lambda)$ for the critical highest weights $\lambda$.

\subsection{Multiplicities}

Suppose again that $M$ is an object in $\CO^f$. Then there are well defined numbers $a_\nu\in \DN$ such that
$$
\cha M=\sum_{\nu\in\hfhd} a_\nu \cha L(\nu)
$$
(cf.~\cite{DGK82}). Note that the sum on the right hand side is, in
general, an infinite sum. We define the {\em multiplicity of $L(\nu)$ in $M$} as
$$
[M:L(\nu)]:= a_\nu.
$$
The matrix $[\Delta(\lambda):L(\mu)]$ is invertible, so the problem of calculating $\cha L(\mu)$ for all $\mu\in\hfhd$ is equivalent to the calculation of the multiplicities $[\Delta(\lambda):L(\mu)]$ for all $\lambda,\mu\in\hfhd$.

\subsection{Block decomposition}

Denote by ``$\sim$'' the equivalence relation on $\hfhd$ that is generated
by $\lambda\sim\mu$ if $[\Delta(\lambda):L(\mu)]\ne 0$. For an equivalence class $\Lambda\in \hfhd/_{\textstyle{\sim}}$ we define the
full subcategory $\CO_\Lambda$ of $\CO$ that consists of all objects
$M$ whose irreducible subquotients are isomorphic to $L(\lambda)$ for some $\lambda\in\Lambda$. We have the following decomposition result.

\begin{theorem}[\cite{DGK82,RCW82}] The functor
\begin{align*}
\prod_{\Lambda\in \hfhd/_{\scriptstyle{\sim}}} \CO_\Lambda & \to \CO, \\
(M_\Lambda) & \mapsto \bigoplus_{\Lambda\in\hfhd/_{\scriptstyle{\sim}}} M_\Lambda
\end{align*}
is an equivalence of categories. 
\end{theorem}

\subsection{The Kac--Kazhdan theorem}
The following theorem gives an explicit description of the equivalence relation ``$\sim$'' on $\hfhd$. Let us denote by ``$\preceq$'' the partial order on $\hfhd$ generated by
$\nu\preceq\lambda$ if there exist $n\in\DN$ and $\beta\in \hR^+$
such that $2(\lambda+\rho,\beta)=n(\beta,\beta)$ and
$\nu=\lambda-n\beta$. In particular, $\nu\preceq\lambda$ implies
$\nu\varle\lambda$, but the converse is not true.

\begin{theorem}[\cite{KK79}] We have $[\Delta(\lambda):L(\nu)]\ne 0$ if and
  only if $\nu\preceq\lambda$.
\end{theorem}
In particular, the equivalence relation ``$\sim$'' is generated by the partial
order relation ``$\preceq$''. For $\Lambda\in\hfhd/_{\textstyle{\sim}}$ set 
\begin{align*}
\hR(\Lambda)  &:=\{\alpha\in\hR\mid
2(\lambda+\rho,\alpha)\in \DZ(\alpha,\alpha) \text{ for some  $\lambda\in\Lambda$}\}\\ 
&=\{\alpha\in\hR\mid
2(\lambda+\rho,\alpha)\in \DZ(\alpha,\alpha) \text{ for all  $\lambda\in\Lambda$}\}
,\\
\hCW(\Lambda) & :=\langle s_{\alpha}\mid \alpha\in
\hR(\Lambda)\cap\hR^{\real}\rangle\subset\hCW.
\end{align*}

If $\Lambda$ is {\em non-critical}, i.e.~ if
$\hR(\Lambda)\subset\hR^{\real}$, then $\Lambda=\hCW(\Lambda).\lambda$
for each $\lambda\in\Lambda$. In this case, the structure of the block
$\CO_{\Lambda}$ can be completely described in terms of the group
$\hCW(\Lambda)$ (which turns out to be a Coxeter group) and the
singularity of its orbit $\Lambda$, cf.~\cite{Fie06}.

\subsection{A duality on $\CO^f$}
We will later  need the following duality functor. For convenience
we only define it on the full subcategory $\CO^f$ of $\CO$ that we
defined earlier. All the
modules that we encounter in  this article belong to $\CO^f$.

For $M\in\CO^f$ we set $M^\star:=\bigoplus_{\lambda\in\hfhd}
\Hom_\DC(M_\lambda,\DC)$. We endow  $M^\star$ with the  action of
$\hfg$ given by
$$
(x.\phi)(m)=\phi(-\omega(x).m),
$$
for $x\in\hfg$, $\phi\in M^\star$ and $m\in M$, where
$\omega\colon\hfg\to\hfg$ is the Chevalley-involution, i.e.~the
involution induced on $\hfg$ by the root system automorphism that
sends $\alpha\in \hR$ to $-\alpha\in \hR$ (cf.~\cite[Section 1.3]{Kac90}). Then $M^\star\in\CO^f$ and we
indeed get a  duality functor on $\CO^f$. It is exact and maps irreducible modules
to irreducible modules. A quick look at characters shows
the following.
\begin{lemma} For each $\lambda\in\hfhd$ we have
  $L(\lambda)^\star=L(\lambda)$.
\end{lemma}

For each $\lambda\in\hfhd$ we denote by 
$$
\nabla(\lambda):=\Delta(\lambda)^\star
$$
the dual of the Verma
module with highest weight $\lambda$. By the above
lemma and the exactness of the duality, $\nabla(\lambda)$ and  $\Delta(\lambda)$ have the same
Jordan-Hölder multiplicities, and $\nabla(\lambda)$ has a simple socle which is isomorphic to
$L(\lambda)$.

\section{The critical hyperplane}\label{sec-crithyp}
In this section we recall the notion of a critical weight for the
affine Kac--Moody algebra $\hfg$. We introduce a shift functor $T$ on
each of the critical blocks  and study the corresponding
{\em graded center} $\CA=\bigoplus_{n\in\DZ}\CA_n$. For a critical weight $\lambda\in\hfhd$  we define  the {\em restricted Verma
  module $\rDelta(\lambda)$} as  the quotient of $\Delta(\lambda)$
by the ideal of $\CA$ generated by $\bigoplus_{n\ne 0}\CA_n$.  We state some fundamental
properties of these modules in Theorem \ref{theorem-ResVerma}. The
proof of this theorem is due to Feigin and Frenkel. Then we
recall  the {\em Feigin--Frenkel conjecture} on the Jordan-Hölder
multiplicities for restricted Verma modules and, finally, state the
main result of this article in Theorem \ref{theorem-MT}.

\subsection{A shift functor}
The defining relations of $\hfg$ show that the derived Lie algebra of
$\hfg$ coincides with the central extension $\tfg$ of the loop algebra, i.e.~
$$
[\hfg,\hfg]=\fg\otimes\DC[t,t^{-1}]\oplus\DC K.
$$
Hence $[\hfg,\hfg]$  is of
codimension one in $\hfg$, so the quotient $\hfg/[\hfg,\hfg]$ is a
one-dimensional Lie algebra. Each character of $\hfg/[\hfg,\hfg]$ gives rise to a
one dimensional module of $\hfg$. In this way we get the simple
modules $L(\zeta\delta)$ for  $\zeta\in\DC$. We have $L(\zeta\delta)\otimes L(\xi\delta)\cong L((\zeta+\xi)\delta)$ for $\zeta,\xi\in \DC$.

Let us define the {\em shift functor} 
\begin{align*}
T\colon \hfg\catmod&\to\hfg\catmod,\\
M&\mapsto M\otimes_\DC L(\delta).
\end{align*}
The action of $\hfg$ on the tensor product is the usual one:
$X.(m\otimes l)=X.m\otimes l+m\otimes X.l$ for $X\in\hfg$, $m\in M$
and $l\in L(\delta)$. The functor $T$ 
 is exact and  preserves the categories $\CO^f$ and $\CO$. Clearly, it is an equivalence on  these categories with inverse $T^{-1}\colon M\mapsto
M\otimes_\DC L(-\delta)$. For $n\in\DZ$ we denote by
$T^n\colon\CO\to\CO$ the $|n|$-fold composition of $T$
or of $T^{-1}$. It is given by the tensor product with the one-dimensional module $L(n\delta)$.

The following lemma is easy to prove (for part (3) use the facts that
$L(\delta)^\star\cong L(\delta)$ and $(M\otimes_\DC
N)^\star=M^\star\otimes_\DC N^\star$).

\begin{lemma}\label{lemma-propT} For each $\lambda\in\hfhd$ we have
\begin{enumerate}
\item $T\Delta(\lambda)\cong \Delta(\lambda+\delta)$,
\item $T L(\lambda)\cong L(\lambda+\delta)$,
\item on $\CO^f$ the functor $T$ commutes with the duality, i.e.~ there is a natural
  equivalence 
$T\circ (\cdot)^\star\cong (\cdot)^\star\circ T$ of functors.
\end{enumerate}
\end{lemma}

Let us denote by $\lambda\mapsto \ol\lambda$ the linear map $\hfhd\to\fhd$ that is dual to the inclusion $\fh\subset\hfh$. 
For a subset $\Lambda$ of $\hfhd$ we denote by $\ol\Lambda$ its image in $\fhd$. 
Now we come to the definition of {\em critical equivalence classes}, {\em critical weights} and {\em critical
  blocks}.

\begin{lemma}\label{lemma-aequcrit} For an equivalence class $\Lambda\in\hfhd/_{\textstyle{\sim}}$ the following are equivalent.
\begin{enumerate}
\item The functor $T$ maps $\CO_{\Lambda}$ into itself. 
\item We have $\lambda+\delta\sim\lambda$ for all $\lambda\in \Lambda$.
\item We have $\DZ\delta\setminus\{0\}=\hR^{\imag}\subset \hR(\Lambda)$.
\item For all $\lambda\in \Lambda$, the restriction of $\lambda$ to the central line
  $\DC K\subset\hfh$ coincides with the restriction of $-\rho$, i.e.~ we have $\langle\lambda,K\rangle=\langle-\rho,K\rangle$.
  
\item The induced dot-action of the affine Weyl group $\hCW$ on  $\ol\Lambda$ in
   factors over an action of the finite Weyl group.
\end{enumerate}
\end{lemma}

\begin{remark} Note that $w.(\lambda+\xi\delta)=w.\lambda+\xi w(\delta)=w.\lambda+\xi\delta$, so $w.\lambda\equiv w.(\lambda+\xi\delta)\mod\delta$ for all $\lambda$, so the dot-action indeed induces an action of $\hCW$ on the quotient $\hfhd/\DC\delta$ and the statement (5) above makes sense.
\end{remark}

If the above conditions on $\Lambda$ are satisfied, we say that $\Lambda$ is a {\em critical equivalence class}. In this case we call each element $\lambda$ of $\Lambda$ a {\em critical weight}, or {\em of critical level}. We let $\hfhd_{\crit}$ be the set of weights of critical level, i.e.~the union of the critical equivalence classes. Condition (4) above shows that $\hfhd_{\crit}$ is an affine hyperplane in $\hfhd$. It is called the {\em critical hyperplane}. 

\begin{proof} From the definition of the blocks, the exactness of $T$ and Lemma \ref{lemma-propT} we deduce the equivalence of (1) and (2).

Note that $(\alpha+n\delta)^\vee=\alpha^\vee+n K^\prime$, where $K^\prime$ is a non-zero multiple of $K$. Hence, $\lambda\in\hfhd$ has the property that $s_{\alpha+n\delta}.\lambda\equiv s_{\alpha+m\delta}.\lambda\mod\delta$ for all $m,n\in\DZ$ if and only if $\langle\lambda+\rho,K\rangle=0$. Hence (4) and (5) are equivalent.

As $(\delta,\delta)=0$ we have $\DZ\delta\setminus\{0\}\subset\hR(\Lambda)$ if and only if $(\lambda+\rho,\delta)=0$ for all $\lambda\in\Lambda$. The latter equation is equivalent to $\langle\lambda+\rho,K\rangle=0$. Hence (3) and (4) are equivalent. 

Clearly, (3) implies (2). On the other hand, suppose that (2) holds, but (3) does not hold. Then  $\hR^{\imag}\cap\hR(\Lambda)=\emptyset$, and $\Lambda$ is an orbit of $\hCW(\Lambda)$ under the dot-action. So $\lambda,\lambda+\delta\in\Lambda$ implies that $\lambda+\delta$ and $\lambda$ are contained in the same $\hCW$-orbit. The invariance of the bilinear form then yields $(\lambda+\delta+\rho,\lambda+\delta+\rho)=(\lambda+\rho,\lambda+\rho)$ which implies, as $(\delta,\delta)=0$, that $(\lambda+\rho,\delta)=0$, hence $\langle\lambda+\rho,K\rangle=0$, hence (4), which is equivalent to (3). Hence we have a contradiction.  So (2) implies (3).
\end{proof}

\subsection{The action of $\hCW$ in the critical hyperplane}

Let us fix a critical equivalence class
$\Lambda\subset\hfhd_{\crit}$. Then we have $\delta\in \hR(\Lambda)$, so
$(\lambda+\rho,\delta)\in\DZ(\delta,\delta)=0$ for some (all)
$\lambda\in\Lambda$. Hence, if $\alpha\in \hR(\Lambda)$, then
either $\alpha=-\delta$ or $\alpha+\delta\in\hR(\Lambda)$. 
If we set $R(\Lambda):=R\cap \hR(\Lambda)$,
then 
$$
\hR(\Lambda)=\{\alpha+n\delta\mid \alpha\in R(\Lambda),
n\in\DZ\}\cup\{n\delta\mid n\in\DZ,n\ne 0\}.
$$
Moreover, we have $s_{\alpha+n\delta}\in \hCW(\Lambda)$ if and only if
$s_{\alpha}\in\hCW(\Lambda)$. We set $\CW(\Lambda)=\hCW(\Lambda)\cap
\CW$. Then $\hCW(\Lambda)$ is the affinization of $\CW(\Lambda)$.

Let $\alpha\in R(\Lambda)$. We now define a bijection
$\alpha\uparrow\cdot\colon\Lambda\to\Lambda$, following \cite{AJS94}. For
$\lambda\in\Lambda$ let $\alpha\uparrow\lambda$ be the minimal element
in $\{s_{\alpha,n}.\lambda\mid n\in\DZ,s_{\alpha,n}.\lambda\ge \lambda\}$. Note that $\alpha\uparrow\lambda=\lambda$ if $\lambda$ is contained in the reflection hyperplane corresponding to some $s_{\alpha,n}$. We
denote by $\alpha\downarrow\cdot\colon\Lambda\to\Lambda$ the
inverse map.

\begin{definition}  We say that
\begin{enumerate}
\item
  $\Lambda$ is {\em generic}, if $\ol{\Lambda}$ contains only one element,
  \item $\Lambda$ is {\em subgeneric}, if $\ol \Lambda$ contains exactly two elements,
\item $\Lambda$ is {\em regular}, if for some (all) $\nu$ in $\Lambda$ we have that
  $w.\nu=\nu$ implies $w=e$ for all $w\in\CW$. 
\end{enumerate}
\end{definition}
Note that part (3) of the above definition refers to the action of the
{\em finite} Weyl group only.  If $\Lambda$ is subgeneric, then there is a unique finite positive root $\alpha$ such that $\ol\Lambda=\{\ol\lambda,s_\alpha.\ol\lambda\}$. Set $R(\Lambda)^+:=R(\Lambda)\cap R^+$.
\begin{definition}  Let $\nu\in \Lambda$. We say that
\begin{enumerate}
\item  $\nu$ is {\em dominant}, if $\langle\nu+\rho,\alpha^\vee\rangle\ge 0$ for all $\alpha\in R(\Lambda)^+$.
\item $\nu$ is {\em anti-dominant}, if $\langle\nu+\rho,\alpha^\vee\rangle\le 0$ for all $\alpha\in  R(\Lambda)^+$.
\end{enumerate}
\end{definition}

\subsection{The graded center of  a critical block}\label{sec-gradedcenter} Since $\Lambda$ is
supposed to be critical, we can consider the functor $T$  as an auto-equivalence on the block $\CO_\Lambda$.

Let $n\in\DZ$ and let $z$ be a natural transformation from the functor $T^n$ on $\CO_\Lambda$ to the identity functor $\id$ on $\CO_\Lambda$. Note that $z$ associates with any $M\in\CO_\Lambda$ a homomorphism $z^M\colon T^nM\to M$ in such a way that for any homomorphism $f\colon M\to N$ in $\CO_\Lambda$ the diagram

\centerline{
\xymatrix{
T^nM\ar[r]^{T^n f} \ar[d]_{z^M}&T^nN\ar[d]^{z^N}\\
 M\ar[r]^{f}& N
}
}
\noindent
commutes. 

Denote by  $\CA_n=\CA_n(\Lambda)$ the complex vector space of all natural transformations $z$ as above such that $z^{T^lM}=T^l z^M\colon T^{n+l}M\to T^l M$ for all $M\in\CO_\Lambda$ and $l\in\DZ$. There is a bilinear map
\begin{align*}
\CA_n\times\CA_m&\to\CA_{m+n}\\
(z_1,z_2)&\mapsto (M\mapsto z_1^M\circ(T^n z_2^M))
\end{align*}
that makes $\CA=\CA(\Lambda):=\bigoplus_{n\in \DZ} \CA_n$ into a graded $\DC$-algebra. It is associative, commutative and unital.

\subsection{Restricted Verma modules}

Let $\lambda\in \Lambda$ and $n\in \DZ$. Each $z\in \CA_n$ defines a homomorphism
$$
z^{\Delta(\lambda)}\colon T^n\Delta(\lambda)\cong\Delta(\lambda+n\delta)\to \Delta(\lambda).
$$
Each such homomorphism is zero if $n>0$.  Define $\Delta(\lambda)^-\subset \Delta(\lambda)$ as the submodule generated by the images of all the homomorphisms $z^{\Delta(\lambda)}$ for $z\in \CA_n$ and $n<0$.

\begin{definition} The quotient 
$$
\rDelta(\lambda) := \Delta(\lambda)/\Delta(\lambda)^-
$$
is called the {\em restricted Verma module} of highest weight $\lambda$. 
\end{definition}

Consider the formal power series 
$$
\prod_{j\ge 1}(1-q^j)^{-\rank\,
  \fg}=\prod_{j\ge1}(1+q^j+q^{2j}+\dots)^{\rank\,\fg},
$$
and let us define for $n\in\DZ$ the number $p(n)\in\DN$ as the
coefficient of $q^n$ in the above series.

We denote by $\hfn_+:=\bigoplus_{\alpha\in \hR^+}\hfg_\alpha$ the
subalgebra of $\hfg$ corresponding to the positive affine roots. For a
$\hfg$-module $M$ we denote by $M^{\hfn_+}$ the set of $\hfn_+$-invariant
vectors. It is a $\hfh$-submodule of $M$, hence we can also define its
weight spaces $M^{\hfn_+}_\nu$ for $\nu\in\hfhd$.

The following theorem lists the most important properties of the restricted Verma modules. The proofs of the following statements (2), (3) and (4), as well as the main step in the proof of (1), are due to Feigin and Frenkel. We recall the main arguments in Section \ref{sec-FFcenter}.

\begin{theorem} \label{theorem-ResVerma} Let $\lambda\in \hfhd_{\crit}$.
\begin{enumerate}
\item The map $\CA_{n}\to \Hom(T^{n}\Delta(\lambda), \Delta(\lambda))$ is surjective for all $n\in \DZ$.
\item We have $\dim \Hom(T^n\Delta(\lambda), \Delta(\lambda))= p(-n)$ for all $n\in \DZ$.
\item We have 
$$
\cha \rDelta(\lambda)=e^{\lambda}\prod_{\alpha\in \hR^{+}\cap \hR^{\real}}(1+e^{-\alpha}+e^{-2\alpha}+\dots).
$$ 
\item We have $\rDelta(\lambda)^{\hfn_+}_{\lambda-n\delta}=\{0\}$ for $n\ne 0$. 
\end{enumerate}
\end{theorem}

\subsection{A conjecture}\label{subs-FFconj} The character of $L(\lambda)$ for a critical
highest weight $\lambda$ is not yet known in general. But we have a
formula for the characters
of the restricted Verma modules and the simple characters can
be calculated once the Jordan--Hölder multiplicites
$[\Delta(\lambda):L(\mu)]$ for critical weights $\lambda$ and $\mu$
are determined. In the following we state a conjecture
that gives a formula for these multiplicities in terms of periodic
Kazhdan--Lusztig polynomials. It is due to Feigin and Frenkel.

Let $\Lambda$ be a critical equivalence class. Then $\hCW(\Lambda)$ is the affinization of the finite Weyl group $\CW(\Lambda)$. So we can think of $\hCW(\Lambda)$ as a group of affine transformations on a vector space. The connected complements of the affine reflection hyperplanes are called alcoves, and the set of alcoves is a principal homogeneous set for the action of $\hCW(\Lambda)$. Let $A_e$ be the unique alcove in the dominant Weyl chamber that contains the origin in its closure. Then the map $w\mapsto A_w:=w(A_e)$ gives a bijection between $\hCW(\Lambda)$ and the set of alcoves. In  \cite{Lus80}
Lusztig defined for  alcoves $A$ and $B$ a {\em periodic
  polynomial} $p_{A,B}\in\DZ[v]$ (we use the normalization and notation of Soergel, cf.~ \cite{Soe97} and \cite{Fie07}). Denote by $w_\Lambda\in\CW(\Lambda)$ the longest element.

\begin{conjecture} \label{conj-FFC} Let $\Lambda$ be a critical equivalence class.
\begin{enumerate}
\item {\em The restricted linkage principle:} For $\lambda,\mu\in
  \Lambda$ we have that $[\rDelta(\lambda):L(\mu)]\ne 0$ implies  $\mu\in\hCW(\Lambda).\lambda$.
\item {\em The restricted Verma multiplicities:} Let $\lambda\in \Lambda$ be regular and dominant and $w\in \hCW(\Lambda)$. Suppose that for all $x,x^\prime\in \hCW(\Lambda)$ with $p_{A_{w_\Lambda x},A_{w_\Lambda w}}(1)\ne 0$ and $p_{A_{w_\Lambda x^\prime},A_{w_\Lambda w}}(1)\ne 0$ and $x\ne x^\prime$ we have $x.\lambda\ne x^\prime.\lambda$. Then
$$
[\rDelta(w.\lambda):L(x.\lambda)]=p_{A_{w_\Lambda w},A_{w_\Lambda x}}(1)
$$
for all $x\in \hCW(\Lambda)$. 
\end{enumerate}
\end{conjecture}

This conjecture is closely related to an anticipated  relation between
representations of a small quantum group, the topology of
semi-infinite flag manifolds and the restricted critical level
representations of an affine Kac--Moody algebra, cf.~\cite{Lus91}. We prove part (1) of the above conjecture in  \cite{AF2}.

\subsection{The main result}  

In the following theorem we summarize the main results of this article.

\begin{theorem}\label{theorem-MT} Suppose that $\Lambda\subset\hfhd$
  is a critical equivalence class.  
\begin{enumerate}
\item If $\nu\in\Lambda$ is 
anti-dominant, then for all
  $w\in\CW(\Lambda)$ and $n\ge 0$ we have
$$
[\Delta(w.\nu):L(\nu-n\delta)]=p(n).
$$
\item If $\Lambda$ is subgeneric and $\nu\in\Lambda$ is dominant, then
  we have for all $n\ge 0$
$$
[\Delta(\nu):L(\nu-n\delta)] = [\Delta(\alpha\uparrow \nu):L(\nu-n\delta)] =  p(n),
$$
where $\alpha$ is the unique positive finite root with $\ol\Lambda=\{\ol\lambda,s_\alpha.\ol\lambda\}$.
\end{enumerate}
\end{theorem}
Let us restate the above results in terms of restricted Verma modules. 
\begin{corollary} Let $\Lambda\subset\hfhd$ be a critical equivalence class.
\begin{enumerate}
\item Suppose that  $\nu\in \Lambda$ is anti-dominant. Then for any
  $\gamma\in\Lambda$ we have
$$
[\rDelta(\gamma):L(\nu)]=
\begin{cases}
1 \text{ if $\gamma\in\CW(\Lambda).\nu$}, \\
0 \text{ else.}
\end{cases}
$$
In particular, Conjecture \ref{conj-FFC} holds for the anti-dominant multiplicities.
\item 
If $\Lambda$ is subgeneric, then we have for all $\gamma,\nu\in\Lambda$
$$
[\rDelta(\gamma):L(\nu)]=
\begin{cases}
1 &\text{ if $\nu\in\{\gamma,\alpha\downarrow\gamma\}$},\\
0 &\text{ else,}
\end{cases}
$$
where $\alpha\in R^+$ is such that $\ol\Lambda=\{\ol\lambda,s_{\alpha}.\ol\lambda\}$. In particular, Conjecture \ref{conj-FFC} holds in the subgeneric cases.
\end{enumerate}
\end{corollary}

In the following section  we recall the results  of
Feigin and Frenkel on the center at the critical level and deduce
Theorem \ref{theorem-ResVerma}.  In Section \ref{sec-Proj} we study
the structure of projective objects in a critical block
$\CO_\Lambda$. In particular, we provide some results on the action
of the graded center $\CA$ on a projective object. In  Section
\ref{sec-BRST} we introduce the BRST-cohomology functor and recall the
main Theorem of \cite{Ara07}. In Section \ref{sec-proofofMT} we use
the results of Sections \ref{sec-FFcenter}, \ref{sec-Proj} and
\ref{sec-BRST} to  give a proof of Theorem \ref{theorem-MT}.

\section{The Feigin--Frenkel center}\label{sec-FFcenter}
In this section we recall the fundamental
results on the Feigin-Frenkel center \cite{FeiFre92}
at the critical level.
The main references are the textbooks
\cite{FreBen04} and \cite{Fre07}.

\subsection{The universal affine vertex algebra
at the critical level}
Set 
\begin{align*}
V^{\crit}(\fing) := U(\affg)\*_{U(\fing\*_{\C} \C[t]\+ \C K \+ \C D)}\C_{\crit},
\end{align*}
where $\C_{\crit}$ is 
the one-dimensional
representation 
of $\fing\*_{\C}\C[t]\+ \C K\+ \C D$
on which 
$\fing\*_{\C}\C[t]\+ \C D$
acts trivially and 
$K$ acts as  multiplication by the critical value $\langle-\rho,K\rangle$.
The space 
 $V^{\crit}(\fing)$
has a
natural structure of a {\em vertex algebra},
and is called the
 {\em universal affine vertex algebra associated with
$\fing$} at the critical level (see e.g., \cite[\S 4.9]{Kac98}).
Let
\begin{align*}
Y(?,z):V^{\crit}(\fing) & \ra \End_{\C} V^{\crit}(\fing)[[z,z\inv]], \\
 a& \mapsto a(z)  =\sum_{n\in \Z}a_{(n)}z^{-n-1}
\end{align*}
be the state-field correspondence.
The  map
$Y(?,z)$
 is
 uniquely determined by the condition
\begin{align*}
 Y((x\* t\inv)\vac,z)=\sum_{n\in \Z}(x\* t^n)z^{-n-1}
\quad \text{for $x\in \fing$,}
\end{align*}
where $\vac$ is the vacuum vector $1\* 1$. 

The vertex algebra $\Vcr$
is graded by the Hamiltonian $-D$.
If $a\in \Vcr$ is an eigenvector of $-D$,
its eigenvalue is called the {\em conformal weight}
and is denoted by $\Delta_a$.
We denote by  $\partial$
 the translation operator. It satisfies
\begin{align*}
 Y(\partial a,z)=[\partial,Y(a,z)]=\frac{d}{dz}Y(a,z).
\end{align*} 
\subsection{The Feigin-Frenkel center}
\newcommand{\FFC}{\mathfrak{z}(\affg)}
Let $\FFC$ be the center of the
vertex algebra $V^{\crit}(\fing)$:
\begin{align*}
\FFC=\{a\in V^{\crit}(\fing)\mid 
[a_{(m)}, v_{(n)}]=0
\text{ for all $v\in V^{\crit}(\fing)$, $m,n\in \Z$}
\}.
\end{align*}
One has
\begin{align*}
 \FFC&=\{a\in \Vcri \mid  v_{(n)}a=0
\text{ for all $v\in \Vcri$, $n\geq 0$}
\},\\
&=V^{\crit}(\fing)^{G [[t]]},
\end{align*}
where $G$ is the adjoint group of $\fing$
and
 $G[[t]]$ is 
the $\C[[t]]$-points  of $G$.

Let $\{U_p(\fing\*_{\C}\C[t\inv]t\inv)\}$ be the standard filtration of
$U(\fing\*_{\C}\C[t\inv]t\inv)$ and set
\begin{align*}
 F_p V^{\crit}(\fing):=U_p(\fing\*_\C{\C}[t\inv]t\inv)\cdot \mathbf{1}
\subset \Vcri.
\end{align*}
This defines a filtration of a vertex algebra.
Let 
$\gr\, \Vcri$ be 
the associated
graded  vertex algebra:
$
\gr\, \Vcri=\bigoplus_{p}F_{p}V^{\crit}(\fing)
/F_{p-1}V^{\crit}(\fing)
$.
It is a commutative vertex algebra
and
one
has
\begin{align*}
\gr\, \Vcri\cong S(\fing[t\inv]t\inv)\cong\C[\fing_{\infty}]
\end{align*}
as differential rings\footnote{A commutative vertex algebra
$V$ is naturally a commutative ring 
with a derivation 
by the multiplication
$(a,b)\mapsto a_{(-1)}b$
and the derivation $\partial a_{(-n)}=n a_{(-n-1)}$.}
and $G[[t]]$-modules,
where 
$\fing_{\infty}$ is the infinite jet scheme of
$\fing$ (cf.\ \cite{EisFre01})
and $\fing$ is identified with $\fing^\star$.
Below we shall identify $\gr\, V^{\crit}(\fing)$ with $\C[\fing_{\infty}]$.
The natural projection $\fing_{\infty}\ra \fing$
gives the embedding $\C[\fing]\hookrightarrow \C[\fing_{\infty}]$.

Let $\{F_p \FFC\}$ be the induced filtration of $\FFC$,
$\gr\, \FFC$ the associated graded vertex algebra.
Certainly, the image of the natural embedding
$
\gr\, \FFC\hookrightarrow \gr\, V^{\crit}(\fing)
=\C[\fing_{\infty}]
$
is contained in 
$ \C[\fing_{\infty}]^{G[[t]]}$.

 Let $\bar p^{(1)}
,\dots, \bar p^{(\ell)}$,
where $\ell=\rank\, \fing$,
 be  homogeneous
generators
 of
the ring $
\C[\fing]^G\subset \C[\fing_{\infty}]^{G[[t]]}$.
The elements $\bar p^{(i)}_{(-j-1)}:=(\partial)^j \bar p^{(i)
}/j!$ are also $G[[t]]$-invariant
for all $j\geq 0$.

According to \cite{BeiDri96}
(see also \cite{EisFre01}), one has
\begin{align*}
 \C[\fing_{\infty}]^{G[[t]]}
=\C[\bar{p}^{(1)}_{(-j-1)},\dots 
\bar{p}^{(\ell)}_{(-j-1)}]_{j\geq 0}.
 \end{align*}
\begin{theorem}[\cite{FeiFre92}, see also \cite{Fre07}]\label{Th:FF}
The  embedding
\begin{align*}
 \gr\, \FFC \ra
\C[\fing_{\infty}]^{G[[t]]}
\end{align*}
is an isomorphism.
\end{theorem}
By Theorem \ref{Th:FF}
there exist  homogeneous
generators  $p^{(1)},\dots,p^{(\ell)}$
of $\FFC$
whose symbols are $\bar{p}^{(1)},\dots, \bar{p}^{(\ell)}$.
Let $d_i=\deg \bar {p}^{(i)}-1$,
so that
$d_1,\dots,d_{\ell}$ are the exponents of $\fing$.
The conformal weight of $p^{(i)}$ is 
by definition $d_i+1$.
We write
\begin{align*}
 Y(p^{(i)},z)&=p^{(i)}(z)=\sum_{n\in \Z}p^{(i)}_{(n)}z^{-n-1}
=\sum_{n\in \Z}p^{(i)}_{n}z^{-n-d_i-1},
\end{align*}
so that
\begin{align*}
[D, p^{(i)}_n]=n p^{(i)}_n.
\end{align*}
One has
\begin{align*}
 [x, p_{(n)}^{(i)}]=0\quad \text{for all $x\in \tfg$.}
\end{align*}
\subsection{The action of the Feigin--Frenkel 
center on objects with  critical level}
For $k\in\DC$ we denote by $\CO_k$ the category $\CO$ at level $k$, i.e.~ the direct summand of the category $\CO$
on which $K$ acts as multiplication with $k$. In particular, we denote by $\CO_{\crit}$ the category $\CO$ at critical level.
Let $M$ be an object of $\CO_{\crit}$.
Then $M$ 
is naturally a graded module over the vertex algebra $\Vcri$,
and hence, it is a graded module over its center $\FFC$.
Thus
$M$ can be viewed as a graded module over the
polynomial ring
\begin{align*}
 \CZ=\C[p^{(i)}_s; i=1,2,\dots, \ell,\ s\in \Z]
=\bigoplus_{n\in \Z}\CZ_n
\end{align*}
in an obvious manner.
Here
$\CZ_n$ is the subspace of $\CZ$
spanned by elements
$p^{(i_1)}_{n_1}\cdots p^{(i_r)}_{n_r}$
with $n_1+\dots +n_r=n$.
Set
\begin{align*}
\RZm =\C[p^{(i)}_n;i=1,\dots ,\ell, n<0]
=\bigoplus_{n\leq 0}\RZm_n\subset \CZ,
\end{align*}
where $\RZm_n=\RZm\cap \CZ_n$.

\begin{theorem}[{\cite{FreGai06}, see also  \cite[Theorem 9.5.3]{Fre07}}]
\label{th:FF-freeness}
For any $\lam\in \affh_{\crit}^{\star}$,
$\Delta(\lam)$ is free over $\RZm$.
Moreover,
the natural map $\CZ_n^-\to\Hom(\Delta(\lambda+n\delta), \Delta(\lambda))$
  is a bijection for all $n\leq 0$.
\end{theorem}

We now construct a natural map 
\begin{align*}
\CZ_n\ra \CA_n.
\end{align*}
Recall that $L(\delta)$ is one-dimensional. So we can choose a
generator $l$ of $L(\delta)$. This gives us, for any $\hfg$-module
$M$, a map $M\otimes_\DC L(\delta)\to M$, $m\otimes l\mapsto m$. Since
$L(\delta)$ is trivial as a $\tfg$-module, this map is a $\tfg$-module
homomorphism. By iteration we get $\tfg$-module homomorphisms $T^n M=M\otimes_\DC
L(\delta)^{\otimes n}\to M$ for all $n\ge 0$. Using the element 
$l^\prime\in L(-\delta)$ that is dual to $l$ with respect to an isomorphism $L(\delta)\otimes_\DC L(-\delta)\to L(0)\cong\DC$ we
analogously get $\tfg$-module homomorphisms $T^{-n} M=M\otimes_\DC L(-\delta)^{\otimes
  n}\to M$. 

Now suppose that $M$ is contained in a critical block of $\CO$. Let $n\in\DZ$ and $z\in\CZ_n$. Then the composition of the map
$T^n M\to M$ constructed above and the action map $z\colon M\to M$
yields now a {\em $\hfg$-module homomorphism} $T^n M\to M$, as it now also commutes with the action of $D$. This 
gives us a natural transformation $T^n \to \id$ and we get an element in
$\CA_n$ that we associate with $z$. Hence we constructed a map $\CZ_n\to
\CA_n$.

\begin{proof}[Proof of Theorem \ref{theorem-ResVerma}]
The action of $\CZ_n$ on $\Delta(\lambda)$ factors over the
action of $\CA_n$. Hence Theorem \ref{th:FF-freeness} implies parts (1) and (2)
of Theorem  \ref{theorem-ResVerma}. 
It also implies  that
$\Delta(\lam)^-$ is spanned by $p^{(i)}_{-n}m$
with $i=1,\dots,\ell$,
$n>0$ and $m\in \Delta(\lam)$.
From the freeness assertion in Theorem \ref{th:FF-freeness}  we deduce
part (3) of Theorem  \ref{theorem-ResVerma}. Finally, its part (4) is
another Theorem of Feigin and Frenkel  \cite{FeiFre90}
which,
for instance,
follows 
from
\cite[Proposition 9.5.1]{Fre07}
and Theorem \ref{th:FF-freeness}.
\end{proof}

%
%
%
%
%
%

\section{Projective objects}\label{sec-Proj}
For a general equivalence class $\Lambda\in\hfhd/_{\textstyle{\sim}}$ the block
$\CO_\Lambda$ does not contain 
enough projective objects  (this includes, for example, all  critical
 equivalence classes). However, there is a way to overcome this problem
by restricting the set of possible weights for the modules under
consideration. This means that we have to consider the {\em truncated}
subcategories of $\CO$. 

\subsection{The truncated categories}

Let us fix a (not necessarily critical) equivalence class $\Lambda\in
\hfhd/_{\textstyle{\sim}}$. We use the following notation:
We write $\{\le \nu\}$ for the set $\{\nu^\prime\in\Lambda\mid
\nu^\prime\le \nu\}$ and use the similar notations $\{<\nu\}$, $\{\ge \nu\}$,
etc.~ for the analogously  defined sets. We consider the topology on
$\Lambda$ that is generated by the basic open sets $\{\le\nu\}$ for
$\nu\in\Lambda$. Hence a subset $\CJ$ of $\Lambda$ is open in this
topology if and only if for all $\nu,\nu^\prime\in\Lambda$ with
$\nu^\prime\le \nu$, $\nu\in\CJ$ implies 
$\nu^\prime\in\CJ$.

\begin{definition} For an open subset $\CJ\subset\Lambda$ we denote by $\CO_\Lambda^{\CJ}\subset\CO_\Lambda$ the full
subcategory of objects $M$ with the property that each of its subquotients is isomorphic to $L(\lambda)$ for some $\lambda\in \CJ$. 
\end{definition}
Note that $\CO_\Lambda^\CJ$ is an abelian category and that
$\Delta(\lambda)$ and $L(\lambda)$ are
contained in $\CO_\Lambda^{\CJ}$ if and only if $\lambda\in\CJ$. We
write $\CO_\Lambda^{\varle\nu}$ instead of $\CO_\Lambda^{\{\varle\nu\}}$.

\subsection{Submodules, quotients and subquotients} 
Let $\Lambda\subset\hfhd$ be an equivalence class and $\CJ$ an open
subset of $\Lambda$.
In this section we construct a left adjoint functor $M\mapsto M^{\CJ}$
to the inclusion functor $\CO_{\Lambda}^\CJ\to\CO_\Lambda$. Let us
denote by $\CI=\Lambda\setminus \CJ$ the closed complement of $\CJ$.

\begin{definition} Let $M$ be an object in $\CO_\Lambda$.
\begin{enumerate}
\item We define  $M_{\CI}\subset M$ as the submodule generated by the
  weight spaces $M_\nu$ with $\nu\in\CI$.
\item We define $M^{\CJ}$ as the quotient of $M$ by the submodule $M_{\CI}$.
\end{enumerate}
\end{definition}

Obviously, the map $M\mapsto M^\CJ$ defines a functor from
$\CO_\Lambda\to\CO_\Lambda^\CJ$ that is left adjoint to the
inclusion. We write $M^{\le\nu}$ instead of $M^{\{\le\nu\}}$ and $M^{<\nu}$
instead of $M^{\{<\nu\}}$, etc. If
$\CI^\prime\subset\CI\subset\Lambda$ are closed subsets, then there is
a natural inclusion $M_{\CI^\prime}\subset M_\CI$. 
If $\CJ^\prime\subset\CJ\subset\Lambda$ are open subsets, then there is
a canonical surjective map $M^\CJ\to M^{\CJ^\prime}$.

In addition to the submodules and quotient modules defined above we
will also need the following subquotient modules that are associated
with locally closed subsets. For any subset $\CK$ of $\Lambda$ define 
$$
\CK_-:=\bigcup_{\lambda\in \CK} \{\le \lambda\} \text{ and } \CK_+:=\bigcup_{\lambda\in \CK}\{\ge \lambda\}.
$$
Note that $\CK_+$ is closed and $\CK_-$ is open, and $\CK$ is locally closed if $\CK=\CK_-\cap\CK_+$.

\begin{definition} Suppose that $\CK$ is locally closed. We define $M_{[\CK]}$ as the image of the canonical decomposition
$M_{\CK_+}\inj M \sur M^{\CK_-}$.
\end{definition}
If $\CK=\{\lambda,\dots,\mu\}$, we write $M_{[\lambda,\dots,\mu]}$
instead of $M_{[\{\lambda,\dots,\mu\}]}$.

\subsection{Modules admitting a Verma flag}

\begin{definition} We say that $M\in\CO_\Lambda$ {\em admits a Verma flag}
  if there is a finite filtration
$$
0=M_0\subset M_1\subset \dots\subset M_n=M
$$
such that for all $i=1,\dots,n$ the quotient $M_i/M_{i-1}$ is isomorphic to a Verma module $\Delta(\mu_i)$ for some $\mu_i\in \hfhd$. 
\end{definition}
For each $\mu\in\hfhd$ the multiplicity 
$$
[M:\Delta(\mu)]:=\# \{i\mid \mu=\mu_i\}
$$
is independent of the chosen filtration.
It is a well-known fact that $\Ext^1(\Delta(\mu),\Delta(\lambda))\ne 0$ implies that $\lambda>\mu$. Hence we can find a filtration for $M$ as in the definition such that $\mu_i>\mu_j$ implies $i<j$. From this one easily gets the following lemma.

\begin{lemma} \label{lemma-Vermasubquots} Suppose that $M$ admits a Verma flag and suppose that $\CJ\subset\hfhd$ is open, $\CI\subset\hfhd$ is  closed and $\CK\subset\hfhd$ is locally closed. Then $M^\CJ$, $M_\CI$ and $M_{[\CK]}$ admit a Verma flag and for the multiplicities we have
\begin{align*}
[M^\CJ:\Delta(\mu)]&=
\begin{cases}
[M:\Delta(\mu)]&\text{ if $\mu\in\CJ$}\\
0&\text{ else,}
\end{cases}
\\
[M_\CI:\Delta(\mu)]&=
\begin{cases}
[M:\Delta(\mu)]&\text{ if $\mu\in\CI$}\\
0&\text{ else,}
\end{cases}
\\
[M_{[\CK]}:\Delta(\mu)]&=
\begin{cases}
[M:\Delta(\mu)]&\text{ if $\mu\in\CK$}\\
0&\text{ else}
\end{cases}
\end{align*}
for all $\mu\in\hfhd$.

\end{lemma}

\subsection{Projective objects}

We say that  $\CJ\subset \Lambda$ is {\em bounded} if for any $\lambda\in\CJ$ the set of $\mu\in\CJ$ with $\mu\ge\lambda$ is finite. 

\begin{theorem}[\cite{Fie03}, cf.~ also \cite{AF2}] Suppose that  $\CJ\subset \Lambda$ is open and bounded and let $\lambda\in \CJ$.
\begin{enumerate}
\item There exists an (up to isomorphism unique) projective
  cover $P^{\CJ}(\lambda)$ of $L(\lambda)$ in $\CO^{\CJ}_\Lambda$. 
\item If  $\CJ^\prime\subset \CJ$ is an open subset and $\lambda\in\CJ^\prime$,  then
  we have  an isomorphism
  $(P^{\CJ}(\lambda))^{\CJ^\prime}\cong P^{\CJ^\prime}(\lambda)$.
\item The object $P^{\CJ}(\lambda)$ admits a Verma flag and  for each $\gamma\in \Lambda$ we have the BGG-reciprocity formula  
$$
\left(P^{\CJ}(\lambda):\Delta(\gamma)\right)=
\begin{cases}
0 & \text{if $\gamma\not\in\CJ$,} \\
[\nabla(\gamma):L(\lambda)] & \text{if $\gamma\in \CJ$.}
\end{cases}
$$
\end{enumerate}
\end{theorem}

By Lemma \ref{lemma-Vermasubquots} and the theorem above, if $\nu,\lambda\in\Lambda$ are such that $\nu\ge
\lambda$, then the module $P^{\CJ}(\lambda)_{[\nu]}$ does not depend on
the open set $\CJ$ as long as $\nu\in\CJ$. We denote this object by
$P(\lambda)_{[\nu]}$.

Let $\CJ\subset\Lambda$  be open and bounded and $\lambda\in \CJ$. Suppose that $\nu\in \CJ$ is a maximal element. Then we can consider $P(\lambda)_{[\nu]}$ as a subspace in $P^{\CJ}(\lambda)$. 
We will need the following result later.
\begin{lemma}\label{lemma-HomdualVerma} Under the above assumptions, the restriction map 
$$
\Hom(P^{\CJ}(\lambda), \nabla(\nu))
 \to \Hom(P(\lambda)_{[\nu]}, \nabla(\nu)) 
$$
is a bijection.
\end{lemma}
\begin{proof} As a first step we show that the map is
  injective. Suppose that $f\in\Hom(P^{\CJ}(\lambda), \nabla(\nu))$ 
  is contained in the kernel. Then $P(\lambda)_{[\nu]}$ is contained
  in the kernel of $f$, hence $f$ factors over the map
  $P^{\CJ}(\lambda)\to (P^{\CJ}(\lambda))^{\CJ\setminus\{\nu\}}\cong P^{\CJ\setminus\{\nu\}}(\lambda)$. But then it cannot
  contain the simple socle $L(\nu)\subset \nabla(\nu)$ in its image,
  so $f=0$. 

  Now we have 
\begin{align*}
\dim\Hom(P^{\CJ}(\lambda), \nabla(\nu)) &=
[\nabla(\nu):L(\lambda)]\\
&= (P^{\CJ}(\lambda):\Delta(\nu))\\
&= (P(\lambda)_{[\nu]}:\Delta(\nu))\\
&=\dim \Hom(P(\lambda)_{[\nu]},\nabla(\nu)).
\end{align*}
The last equation holds since $P(\lambda)_{[\nu]}$ is isomorphic to a direct sum of copies of $\Delta(\nu)$ and $\dim\Hom(\Delta(\nu),\nabla(\nu))=1$.  Since the map referred to in the lemma is  injective and since the dimensions of its source and its image coincide, it is bijective. 
\end{proof}

\subsection{The Casimir operator}

We call a $\hfg$-module $M$ {\em smooth} if for all $m\in M$ we have
$\hfg_\alpha.m=0$ for all but a finite number of positive affine roots
$\alpha\in\hR$. In particular, each locally $\hfb$-finite $\hfg$-module is smooth. Recall that on the full subcategory $\hfg\catmod^{sm}$ of
smooth representations there is an endomorphism of the identity
functor, the {\em Casimir operator} $C\colon \id\to \id$ (its
construction can be found, for example, in \cite[Section 2.5]{Kac90}). We will
only need the following property of $C$.

\begin{proposition} Let $\lambda\in\hfhd$. Then $C$ acts on
  $\Delta(\lambda)$ as multiplication with the scalar $c_\lambda:=(\lambda+\rho,\lambda+\rho)-(\rho,\rho)\in\DC$.
\end{proposition}
The construction
of $C$ depends on $(\cdot,\cdot)$, hence there is no ambiguity in the
statement of the proposition. 

Let $\Lambda\subset\hfhd$ be a (not necessarily critical) equivalence class. For 
$\lambda,\mu\in \Lambda$ we have  $c_\lambda=c_\mu$, so $\Lambda$ defines a unique scalar  $c_\Lambda\in\DC$ such that $C$
acts by multiplication with $c_\Lambda$ on each Verma module in $\CO_\Lambda$.

Now suppose that 
$\lambda,\mu\in\Lambda$ are such that they form an atom in the
partially ordered set $\Lambda$, i.e.~  suppose that $\lambda< \mu$
and that there is no $\nu\in\Lambda$ with $\lambda<\nu<\mu$. Moreover, assume that $\ol\lambda\ne\ol\mu$. Then the
object $P^{\varle\mu}(\lambda)$ is an extension of the Verma modules
$\Delta(\lambda)$ and $\Delta(\mu)$, each occurring once (this is, by the BGG-reciprocity, equivalent to $[\Delta(\lambda):L(\lambda)]=[\Delta(\mu):L(\lambda)]=1$, which  can be deduced easily from 
the analog of Jantzen's sum formula in the Kac-Moody case,
cf.~\cite{KK79}). Hence we have a short exact sequence
$$
0\to \Delta(\mu)\to P^{\varle\mu}(\lambda)\to\Delta(\lambda)\to 0.
$$

\begin{lemma} The endomorphism $C-c_\Lambda\id\colon
  P^{\varle\mu}(\lambda)\to P^{\varle\mu}(\lambda)$ is non-zero.
\end{lemma}

\begin{proof} For the proof we use deformation theory,
  cf.~\cite{Fie03} and \cite{AF2}. Denote by $A=\DC[[t]]$ the completed polynomial ring in
  one variable and by $Q=\Quot A$ its quotient field. Let us fix $\gamma\in \hfhd$ with the property that $(\gamma+\rho,\lambda)\ne(\gamma+\rho,\mu)$. Let $S=S(\hfh)$
  be the symmetric algebra over $\hfh$ and consider the algebra map
  $\tau\colon S\to A$ that is determined by $\tau(h)=\gamma(h)t$ for all
  $h\in\hfh$. This makes $A$, and hence $Q$, into a local $S$-algebra.

  In \cite{Fie03} we
  constructed the deformed categories $\CO_A$ and $\CO_Q$ as full
  subcategories of $\hfg\otimes_\DC A\catmod$ and $\hfg\otimes_\DC
  Q\catmod$. We showed that the  functors $\otimes_A \DC$ and
  $\otimes_A Q$ induce functors $\CO_A\to\CO$ and
  $\CO_A\to\CO_Q$. The categories $\CO_A$ and $\CO_Q$ contain {\em deformed Verma modules} $\Delta_A(\nu)$ and $\Delta_Q(\nu)$, resp. On their highest weight spaces the Cartan algebra $\hfh$ acts by the character $\nu+\tau$, which is considered as a linear map from $\hfh$ to $A$ and $Q$, resp. In particular, the Casimir operator $C$ acts on $\Delta_Q(\nu)$ as multiplication with $c_{\nu+\gamma t}=(\nu+\gamma t+\rho,\nu+\gamma t+\rho)-(\rho,\rho)$.

  The analogous definition as in the non-deformed case gives
  truncated categories $\CO_A^{\varle\mu}$ and $\CO_Q^{\varle\nu}$
  for all $\nu\in\hfhd$. For 
  any $\nu,\nu^\prime\in\hfhd$ there is a projective
  object $P^{\varle\nu}_A(\nu^\prime)$ in $\CO_A^{\varle\nu}$ such that
  $P^{\varle\nu}_A(\nu^\prime)\otimes_A\DC\cong
  P^{\varle\nu}(\nu^\prime)$. Moreover, we have
$$
\Hom(P^{\varle\nu}_A(\nu^\prime_1),
P^{\varle\nu}_A(\nu_2^\prime))\otimes_A\DC=\Hom(P^{\varle\nu}(\nu_1^\prime),
P^{\varle\nu}(\nu_2^\prime)).
$$

The category $\CO_Q$ is semi-simple, i.e.~ each object is isomorphic
to a direct sum of Verma modules $\Delta_Q(\lambda)$. Now each 
 $P^{\varle\nu}_A(\nu^\prime)\otimes_A Q$ admits a $Q$-deformed Verma flag, hence it splits into a
direct sum of $Q$-deformed Verma modules. The Verma multiplicities of $P^{\varle\nu}_A(\nu^\prime)$, of $P^{\varle\nu}(\nu^\prime)$ and of $P_A^{\varle\nu}(\nu^\prime)\otimes_A Q$ coincide.

Now let $\lambda,\mu\in\Lambda$ be as in the statement of the lemma.  Note that $c_{\lambda+\gamma t}\equiv c_{\mu+\gamma t}\mod t$, but our choice of $\gamma$ implies $c_{\lambda+\gamma t}\ne c_{\mu+\gamma t}$. Let us suppose that the action of $C-c_{\Lambda}\id$ on $P^{\varle\mu}(\lambda)$ was
zero.
From the above, our assumptions on $\lambda$ and $\mu$ and the Jantzen sum formula we get an inclusion
$$
P^{\varle\mu}_A(\lambda)\subset P^{\varle\mu}_A(\lambda)\otimes_A Q\cong\Delta_Q(\lambda)\oplus\Delta_Q(\mu).
$$ 
On each of the modules above the Casimir operator acts. Our assumptions imply that the image of the action of $C-c_\Lambda\id$ on the module on the left is contained in $tP^{\varle\mu}_A(\lambda)$, hence $t^{-1}(C-c_\Lambda\id)$ is a well-defined operator on $P^{\varle\mu}_A(\lambda)$. On the module on the right this operator acts diagonally with eigenvalues in $\DC[[t]]$ which are distinct modulo $t$. Hence $P^{\varle\mu}_A(\lambda)$ decomposes according to the inclusion above, which clearly cannot be the case.
\end{proof}

\begin{lemma}\label{lemma-Casimir} Suppose that
  $\lambda,\mu\in\Lambda$ are as above and that
  $$
  0\to \Delta(\mu)\to M\to X\to 0
  $$
  is a short exact sequence, where $X$ is a module of highest weight $\lambda$. 
  If $C$ acts on $M$ as a scalar, then there is a
  submodule $Y$ of $M$ with highest weight $\lambda$ that maps surjectively onto $X$. 
\end{lemma}
\begin{proof} Since $X$ is a module of highest weight $\lambda$ and since the weights of $M$ are smaller or equal to $\mu$, there is a map $f\colon P^{\varle\mu}(\lambda)\to M$ such that the composition $P^{\varle\mu}(\lambda)\to M\to X$ is surjective. There is a commutative diagram 

\centerline{
\xymatrix{
0 \ar[r] & \Delta(\mu) \ar[r]\ar[d]^{f^\prime} & P^{\varle\mu}(\lambda) \ar[r] \ar[d]^f& \Delta(\lambda) \ar[r]\ar@{>>}[d] &0\\
0 \ar[r] & \Delta(\mu) \ar[r] & M \ar[r] & X \ar[r] &0
}
}
\noindent
with exact rows. In order to prove the lemma it is enough to show that $\Delta(\mu)\subset P^{\varle\mu}(\lambda)$ is in the kernel of $f$, i.e.~ that the left vertical map $f^\prime$  is zero. Since any endomorphism of a Verma module is either zero or injective, it suffices to show that $f^\prime$ is not injective.

By assumption, the Casimir element $C$ acts on $M$ by a scalar, which
has to be $c_\Lambda$.  So $C-c_\Lambda\id$ acts on $M$ by zero, so
$(C-c_\Lambda\id)P^{\varle\mu}(\lambda)$ is in the kernel of $f$. Since
$C-c_\Lambda\id$ acts by zero on $\Delta(\lambda)$, we have
$(C-c_\Lambda\id)P^{\varle\mu}(\lambda)\subset \Delta(\mu)\subset
P^{\varle\mu}(\lambda)$. By the previous lemma,
$(C-c_\Lambda\id)P^{\varle\mu}(\lambda)$ is non-zero, so the kernel of $f^\prime$ is not trivial, hence $f^\prime$ is not injective, hence it must be zero, which is what we wanted to show. 
\end{proof}

\subsection{The action of $\CA$ on projective objects}

Let us fix now a critical equivalence class $\Lambda\subset\hfhd$.
Let $\lambda\in \Lambda$ and  $n\ge  0$. In this section we study the
action of $\CA_n$ on
$P^{\varle\lambda}(\lambda-n\delta)$, i.e.~  we want to study  the map
$$
\CA_n\to \Hom(T^n P^{\varle\lambda}(\lambda-n\delta), P^{\varle\lambda}(\lambda-n\delta)).
$$

For the ease of notation let us fix an identification
$T^nP^{\varle\lambda}(\lambda-n\delta)\cong P^{\varle\lambda+n\delta}(\lambda)$.
 Each map $f\colon P^{\varle\lambda+n\delta}(\lambda)\to
 P^{\varle\lambda}(\lambda-n\delta)$ factors over the map
 $P^{\varle\lambda+n\delta}(\lambda)\to
 P^{\varle\lambda}(\lambda)$. The latter module is isomorphic to
 $\Delta(\lambda)$, hence the image of the map $f$ must be contained
 in $P(\lambda-n\delta)_{[\lambda]}\subset P^{\varle\lambda}(\lambda-n\delta)$. So the map $f$ induces a unique
 map $f^\prime\colon P^{\varle\lambda}(\lambda)\to
P(\lambda-n\delta)_{[\lambda]}$ such that the following diagram
 commutes:

\centerline{
\xymatrix{
T^n P^{\varle\lambda}(\lambda-n\delta)\ar[d]\ar[r]^{f}&P^{\varle\lambda}(\lambda-n\delta)\\
P^{\varle\lambda}(\lambda)\ar[r]^{f^\prime}&P(\lambda-n\delta)_{[\lambda]}.\ar[u]
}
}

The next proposition  is one of the principal technical ingredients in the proof of our main Theorem.
\begin{proposition}\label{prop-actonproj}
Suppose that $n\ge 0$ is such that $[\Delta(\lambda):L(\lambda-n\delta)]=p(n)$. Then the action map
\begin{align*}
\CA_n&\to \Hom(T^nP^{\varle\lambda}(\lambda-n\delta),
P^{\varle\lambda}(\lambda-n\delta))\\
&\cong\Hom(P^{\varle\lambda}(\lambda), P(\lambda-n\delta)_{[\lambda]})
\end{align*}
 is surjective. 
\end{proposition}
\begin{proof} Note that $P(\lambda)^{\varle\lambda}\cong\Delta(\lambda)$
  and that $P(\lambda-n\delta)_{[\lambda]}$ is isomorphic to a direct
  sum of $(P(\lambda-n\delta)_{[\lambda]}:\Delta(\lambda))$-copies of
  $\Delta(\lambda)$. By our assumption and the BGG-reciprocity, this number is $p(n)$. Hence the spaces on the right hand side of our map are of dimension $p(n)$. 

Now the following Lemma \ref{lemma-actonVerma}  shows that the image of the action map
$\CA_n\to \Hom(T^nP^{\varle\lambda}(\lambda-n\delta),
P^{\varle\lambda}(\lambda-n\delta))$ is of dimension $p(n)$. From this we deduce  our
claim.
\end{proof}

\subsection{A duality on $\CA$}

Let $\Lambda\subset\hfhd$ again be a critical equivalence class. 
In this section we define an algebra involution
$$
\dual\colon\CA\to\CA
$$
which maps $\CA_n$ into $\CA_{-n}$. 

Fix $n\in\DZ$ and choose $z\in\CA_n$. We define $\dual z\in\CA_{-n}$
as follows.  Let $M\in \CO^f_\Lambda:=\CO^f\cap\CO_\Lambda$ and let $M^\star\in \CO_\Lambda$
be its restricted dual. Then $z$ defines a homomorphism
$z^{M^\star}\colon T^n M^\star\to M^\star$. The dual of this map is a
homomorphism $\left( z^{M^\star}\right)^\star\colon M\to T^n M$. 

\begin{definition} For $z\in\CA_n$ and $M\in\CO^f_\Lambda$ define the
  map 
$$
(\dual z)^M:= T^{-n}\left( z^{M^\star}\right)^\star\colon T^{-n} M\to M.
$$
\end{definition}
One immediately checks that we get a natural transformation 
$\dual z\colon T^{-n}|_{\CO_{\Lambda}^f}\to \id_{\CO_{\Lambda}^f}$. As $\CO_\Lambda$ is filtered by the truncated categories, and as each indecomposable projective object in a truncated category is also contained in $\CO_{\Lambda}^f$, this induces a natural transformation $\dual z\colon T^n\to \id$ between the functors on the whole block $\CO_\Lambda$, hence  an element in $\CA_{-n}$. 

Now we prove the statement that remained open in the proof of Proposition \ref{prop-actonproj}. Fix $\lambda\in\Lambda$ and $n\ge 0$.

\begin{lemma}\label{lemma-actonVerma} For $z\in\CA_n$  the following holds:
$$
(\dual z)^{\Delta(\lambda)}\ne 0 \text{ if and only if } z^{P^{\varle\lambda}(\lambda-n\delta)}\ne 0.
$$
In particular, the image of the map
$\CA_n\to\Hom(T^nP^{\varle\lambda}(\lambda-n\delta),P^{\varle\lambda}(\lambda-n\delta))$
is of dimension $p(n)$.
\end{lemma}

\begin{proof} By the definition of the duality we have 
$$
(\dual z)^{\Delta(\lambda)}\ne 0 \text{ if and only if } z^{\nabla(\lambda)}\ne 0.
$$

For each homomorphism $g\colon P^{\varle\lambda}(\lambda-n\delta)\to \nabla(\lambda)$ there is a commutative diagram

\centerline{
\xymatrix{
T^n 
P^{\varle\lambda}(\lambda-n\delta)\ar[d]_{T^ng}\ar[rrr]^{z^{P^{\varle\lambda}(\lambda-n\delta)}}&&&
P^{\varle\lambda}(\lambda-n\delta)\ar[d]^{g}\\
T^n\nabla(\lambda)\ar[rrr]^{z^{\nabla(\lambda)}}&&&
\nabla(\lambda).
}
}
The strategy of the proof is the following. Suppose that
$z^{P^{\varle\lambda}(\lambda-n\delta)}\ne 0$. We show that there is
a map $g$ such that the top right composition in the diagram above is
non-zero. From this we deduce that $z^{\nabla(\lambda)}\ne 0$. We show
that $z^{\nabla(\lambda)}\ne 0$ implies
$z^{P^{\varle\lambda}(\lambda-n\delta)}\ne 0$ in a similar way.  

So suppose that $z^{P^{\varle \lambda}(\lambda-n\delta)}\ne 0$. We
have already seen that there is a unique map
$b\colon P^{\varle\lambda}(\lambda)\to P(\lambda-n\delta)_{[\lambda]}$ such
that the following diagram commutes:

\centerline{
\xymatrix{
T^n P^{\varle\lambda}(\lambda-n\delta)\ar[d]\ar[rr]^{z^{P^{\varle\lambda}(\lambda-n\delta)}}&&P^{\varle\lambda}(\lambda-n\delta)\\
P^{\varle\lambda}(\lambda)\ar[rr]^b&&P(\lambda-n\delta)_{[\lambda]}.\ar[u]
}
}
\noindent
Now
$P^{\varle\lambda}(\lambda)$ is isomorphic to $\Delta(\lambda)$ and
$P(\lambda-n\delta)_{[\lambda]}$ is a direct sum of various
copies of $\Delta(\lambda)$. Hence  we can
find a map $g^\prime\colon P(\lambda-n\delta)_{[\lambda]}\to \nabla(\lambda)$
such that the composition $P^{{\varle\lambda}}(\lambda)\stackrel{b}\to
P(\lambda-n\delta)_{[\lambda]}\stackrel{g^\prime}\to
\nabla(\lambda)$ is non-zero. By 
Lemma \ref{lemma-HomdualVerma}, the map $g^\prime$ admits a lift
$g\colon
P^{\varle\lambda}(\lambda-n\delta)\to
\nabla(\lambda)$. Diagrammatically, the situation now looks as follows:

\centerline{
\xymatrix{
T^n P^{\varle\lambda}(\lambda-n\delta))\ar[d]\ar[rr]^{z^{P^{\varle\lambda}(\lambda-n\delta)}}&&P^{\varle\lambda}(\lambda-n\delta)\ar[dr]^{g}&\\
P^{\varle\lambda}(\lambda)\ar[rr]^b&&P(\lambda-n\delta)_{[\lambda]}\ar[u]\ar[r]^{g^\prime}&\nabla(\lambda).
}
}

If we plug in the map $g$ that we just obtained in the first diagram above, then
top right composition is non-zero, hence so is the bottom left composition. In particular, 
$z^{\nabla(\lambda)}\ne 0$. This was the first part of the proof.  

Now suppose that $z^{\nabla(\lambda)}\ne 0$. Then the image of $z^{\nabla(\lambda)}$ contains the unique simple submodule $L(\lambda)$ of
$\nabla(\lambda)$. Since
$T^n P^{\varle\lambda}(\lambda-n\delta)\cong
P^{{\varle\lambda}+n\delta}(\lambda)$ is a projective cover of
$L(\lambda)$ in $\CO^{{\varle\lambda}+n\delta}$, and since
$T^n \nabla(\lambda)\cong \nabla(\lambda+n\delta)$ is contained in the latter category, we
can find a map $g^\prime\colon T^nP^{\varle\lambda}(\lambda-n\delta)\to
T^n \nabla(\lambda)$ such that
$z^{\nabla(\lambda)}\circ g^\prime$ is non-zero. For  $g:=T^{-n}g^\prime$, the bottom left composition in the first diagram  in this proof   is
non-zero, hence so is the top right composition. In particular, $z^{P^{\varle\lambda}(\lambda-n\delta)}\ne 0$. 

The last statement of the lemma follows from the previous result and
Theorem \ref{theorem-ResVerma}.
\end{proof}

\subsection{A variant for the subgeneric cases}

We will also need the following variant of Proposition \ref{prop-actonproj} in the case that $\Lambda$ is
critical and subgeneric. Suppose that $\alpha$ is the positive finite root with $\ol\Lambda=\{\ol\lambda,s_\alpha.\ol\lambda\}$.  Fix $\lambda\in\Lambda$ and $n\ge 0$. We study the
action of $\CA_n$ on the projective cover
$P^{\varle\alpha\uparrow\lambda}(\lambda-n\delta)$, i.e.~we now consider the map
$$
\CA_n\to \Hom(T^nP^{\varle\alpha\uparrow\lambda}(\lambda-n\delta),
P^{\varle\alpha\uparrow\lambda}(\lambda-n\delta)).
$$
Again, let us fix an isomorphism
$T^nP^{\varle\alpha\uparrow\lambda}(\lambda-n\delta)\cong
P^{\varle\alpha\uparrow\lambda+n\delta}(\lambda)$. 
As before we see that each map $f\colon
P^{\varle\alpha\uparrow\lambda+n\delta}(\lambda)\to 
P^{\varle\alpha\uparrow\lambda}(\lambda-n\delta)$ induces a map
$f^\prime\colon P^{\varle\alpha\uparrow\lambda}(\lambda)\to P(\lambda-n\delta)_{[\lambda,\alpha\uparrow\lambda]}
$
such that the following diagram commutes:

\centerline{
\xymatrix{
P^{{\varle\alpha\uparrow\lambda+n\delta}}(\lambda)\ar[d]\ar[r]^{f}&P^{{\varle\alpha\uparrow\lambda}}(\lambda-n\delta)\\
 P^{\varle\alpha\uparrow\lambda}(\lambda)\ar[r]^{f^\prime}& P(\lambda-n\delta)_{[\lambda,\alpha\uparrow\lambda]}.\ar[u]
}
}
\noindent
Note that $P^{\varle\alpha\uparrow\lambda}(\lambda)$ is a non-split
extension of the Verma modules $\Delta(\lambda)$ and
$\Delta(\alpha\uparrow\lambda)$, as $\{\lambda,\alpha\uparrow\lambda\}\subset\Lambda$ is an atom. The module
$P(\lambda-n\delta)_{[\lambda,\alpha\uparrow\lambda]}$ is an extension
of $[\Delta(\lambda):L(\lambda-n\delta)]$ many copies of
$\Delta(\lambda)$ and 
$[\Delta(\alpha\uparrow\lambda):L(\lambda-n\delta)]$ many copies of
$\Delta(\alpha\uparrow\lambda)$. 

Let us assume that
$[\Delta(\alpha\uparrow\lambda):L(\lambda-n\delta)]=[\Delta(\lambda):L(\lambda-n\delta)]=p(n)$.
Then $P(\lambda-n\delta)_{[\lambda,\alpha\uparrow\lambda]}$ is a
direct sum of $p(n)$ non-split extensions of the Verma modules
$\Delta(\lambda)$ and $\Delta(\alpha\uparrow\lambda)$ (the extensions
are non-split by projectivity).

Let us consider the composition 
\begin{align*}
\CA_n&\to \Hom(T^nP^{\varle\alpha\uparrow\lambda}(\lambda-n\delta),
P^{\varle\alpha\uparrow\lambda}(\lambda-n\delta))\\
&\to
\Hom(P(\lambda)_{[\alpha\uparrow\lambda]},
P(\lambda-n\delta)_{[\alpha\uparrow\lambda]}),
\end{align*}
where the last map is induced by the functor
$(\cdot)_{[\alpha\uparrow\lambda]}$.

\begin{proposition}\label{prop-actonprojsub}
Suppose that
$[\Delta(\alpha\uparrow\lambda):L(\lambda-n\delta)]=[\Delta(\lambda):L(\lambda-n\delta)]=p(n)$.
Then the composition $\CA_n\to
\Hom(P(\lambda)_{[\alpha\uparrow\lambda]},
P(\lambda-n\delta)_{[\alpha\uparrow\lambda]})$ constructed above is
surjective.
\end{proposition}
\begin{proof} Consider the homomorphisms

\centerline{
\xymatrix{ &  \Hom(P(\lambda)_{[\alpha\uparrow\lambda]},
P(\lambda-n\delta)_{[\alpha\uparrow\lambda]})\\
\CA_n \ar[r] & \Hom(P(\lambda)_{[\lambda,\alpha\uparrow\lambda]},
P(\lambda-n\delta)_{[\lambda,\alpha\uparrow\lambda]}) \ar[u]\ar[d] 
\\
&\Hom(P(\lambda)_{[\lambda]},
P(\lambda-n\delta)_{[\lambda]}).
}
}
By Proposition \ref{prop-actonproj}, the lower composition is
surjective. But the kernel of the upper  composition is contained in
the kernel of the lower composition, as $P(\lambda-n\delta)_{[\alpha\uparrow\lambda,\lambda]}$ is a direct sum of non-split extensions of $\Delta(\alpha\uparrow\lambda)$ and $\Delta(\lambda)$. Since the spaces on the top and the bottom 
share the same dimension, also the upper composition is surjective.
\end{proof}


\section{The BRST cohomology}\label{sec-BRST}
To prove Theorem \ref{theorem-MT} we need a result from \cite{Ara07},
which we explain below.
\subsection{The BRST cohomology associated with the 
quantized Drinfeld-Sokolov reduction}
Denote by
$\finn_-:=\bigoplus_{\alpha\in R^+}\fing_{-\alpha}$
the nilpotent subalgebra of $\fing$ corresponding to the set of negative
roots.
Let 
$\Psi$ be a non-degenerate character of $\finn_-$
in the sense of Kostant \cite{Kos78}, i.e.~ 
\begin{align*}
\Psi(x)=k(x,e_{prin})
\end{align*}
for some principal nilpotent element $e_{prin}$ of $\fing$ in $\finn_+
:=\bigoplus_{\alpha\in R^+}\fing_{\alpha}$.
We 
extend $\Psi$ to the character $\widehat{\Psi}$
of $\finn_-[t,t\inv]:=\finn_-\* \C[t,t\inv]\subset \affg$
by setting
\begin{align*}
\widehat{\Psi}(u\* t^n)=\Psi(u)\delta_{n,0}
\end{align*}
for $u\in \finn_-$, $n\in \Z$.
Let $\C_{\widehat{\Psi}}$ be the 
corresponding one-dimensional representation of $\finn_-[t,t\inv]$. 

Set
\begin{align*}
 H^i(M):=\BRST{i}{M}
\end{align*}
for $M\in \CO_{\crit}$ and $i\in \Z$.
Here,
$M\* \C_{\widehat{\Psi}}$ is considered as an
$\finn_-[t,t\inv]$-module by the tensor product action,
and
$\BRST{\bullet}{M}$
is the semi-infinite $\finn_-[t,t\inv]$-cohomology \cite{Feu84} with  
coefficients in  
$M\* \C_{\widehat{\Psi}}$.
The cohomology  $H^{\bullet}(M)$
is defined by the semi-infinite analogue of the 
Chevalley-Eilenberg complex: 
Let $\Cl$
be the unital superalgebra
generated by
odd elements  $\psi_{\alpha}(n)$
for $\alpha\in R$, $n\in \Z$,
with the relations
\begin{align*}
\psi_{\alpha}(m)\psi_{\beta}(n)+ \psi_{\beta}(n)\psi_{\alpha}(m)
=\delta_{m+n,0}\delta_{\alpha+\beta,0}.
\end{align*}
Let $\semiwedge{\bullet}$ be the irreducible
representation of $\Cl$ generated by the vector $\vac$
such that
\begin{align*}
 \psi_{\alpha}(n)\vac=0,\quad \text{if }\alpha+n\delta\in \widehat{R}^+.
\end{align*}
The space $\semiwedge{\bullet}$ is graded by {\em charge}:
$\semiwedge{\bullet}=\bigoplus_{i\in \Z}\semiwedge{i}$,
where
the charges of $\vac$, $\psi_{\alpha}(n)$
and $\psi_{-\alpha}(n)$
for $\alpha\in R^+$, $n\in \Z$,
are $0$,
$1$, and $-1$,
 respectively.
Also, we view $\semiwedge{\bullet}$
as an $\affh$-module
on which $h\in \affh$
acts as  $h\mathbf{1}=0$ and $[h,\psi_{\alpha}(n)]=
\bra \alpha+n\delta,h\ket \psi_{\alpha}(n)$.

Let
$M$ be an object of  $\CO$.
Set
\begin{align*}
 C^{\bullet}(M):=M\* \semiwedge{\bullet}=\bigoplus_{i\in \Z}C^i(M),
\quad C^i(M)=M\* \semiwedge{i}.
\end{align*}
Define an odd operator
$Q$ of charge $1$ on $C^{\bullet}(M)$
by 
\begin{align*}
 Q:=&\sum_{\alpha\in R^-\atop n\in \Z}
(x_{\alpha}\* t^{-n}+\widehat{\Psi}(x_{\alpha}\* t^{-n}))\*
 \psi_{\alpha}(n)
\\ &\quad-
 \frac{1}{2} \sum_{\alpha,
\beta,\gamma\in R^-\atop
k,l\in \Z}
c_{\alpha,\beta}^{\gamma}
\id_M\* \psi_{-\alpha}(-k)
\psi_{-\beta}(-l)\psi_{\gamma}(k+l),
\end{align*} 
where $x_{\alpha}$ is a (fixed) root vector in $\fing_{\alpha}$ for any $\alpha\in R^-$
and
 $[x_{\alpha},x_{\beta}]=
\sum_{\gamma}c_{\alpha,\beta}^{\gamma}x_{\gamma}$.
The operator $Q$ is well-defined because $M\in \CO$.
One has 
\begin{align*}
Q^2=0.
\end{align*}
Therefore, 
$(C^{\bullet}(M),Q)$ is a cochain complex.
The space $H^{\bullet}(M)$ is
by definition
 the 
cohomology of the complex $(C^{\bullet}(M),Q)$.

Set
$C^{\bullet}(M)_d:=\{c\in C^{\bullet}(M)\mid
(D\* 1+1\* D) c= dc \}$.
One has 
$C^{\bullet}(M)=\bigoplus_{d\in \C}C^{\bullet}(M)_d$.
Because the operator  $Q$
obviously preserves each subspace $C^{\bullet}(M)_d$,
$H^{\bullet}(M)$ is also graded by the diagonal action of $D$:
\begin{align*}
H^{\bullet}(M)=\bigoplus_{d\in \C}H^{\bullet}(M)_d.
\end{align*}

\subsection{The functor $F$}
Let $p$ be an element of $\FFC$.
For each $n\in \Z$,
the operator  $p_{(n)}\* 1$ 
commutes with  the action of $Q$ on $C^{\bullet}(M)$. 
Therefore, 
for each $i\in \Z$,
$H^{i}(M)$ is 
naturally a graded module
over  the commutative vertex algebra
$\FFC$,
and thus can be considered as a graded $\mathcal{Z}$-module.
Denote by $F$ the functor
\begin{align*}
\CO_{\crit}\ra  \mathcal{Z}\operatorname{-Mod},
\quad M\mapsto H^0(M),
\end{align*}
where $\mathcal{Z}\operatorname{-Mod}$ is the 
category of graded $\mathcal{Z}$-modules.

Set 
\begin{align*}
\cha_q F(M):=\sum_{d\in \C}q^{-d} \dim_{\C}F(M)_d
\end{align*}
for a finitely generated object $M$
of $\CO_{\crit}$,
where $F(M)_d=H^0(M)_d$.

\begin{theorem}[\cite{Ara07}]\label{Th:main-of-Ara}$ $

 \begin{enumerate}
  \item One has $H^i(M)=0$ for all $i\ne 0$
and $M\in \CO_{\crit}$.
In particular, the functor $F$ is exact.
\item 
Let  $\lam\in \affh^{\star}_{\crit}$.
One has the following.
\begin{align*}
&\cha_q F(\Delta(\lam))=q^{-\bra \lam,D\ket}\prod\limits_{j\geq 1}(1-q^j)^{-\rank\, \fing},
\\
 &\cha_q F(L(\lam))=\begin{cases}
		    q^{-\bra \lam,D\ket}&\text{if $\lam$ is
		    anti-dominant},\\
0&\text{otherwise.}
		   \end{cases}
\end{align*}
 \end{enumerate}
\end{theorem}
\begin{Rem}
In general,
the correspondence $M\mapsto H^0(M)$ defines a functor 
from $\CO_k$ to the
category of graded modules over the 
$W$-algebra $\mathscr{W}^k(\fing)$
associated with $\fing$ at level $k$,
which coincides \cite{FeiFre92}
with $\FFC$ if the level $k$ is critical.
In \cite{Ara07},
it was proved that
 the functor $H^0(?)$ is exact 
 and 
$H^0(L(\lam))$
is zero or irreducible
for any $\lam$ at any level $k$.
\end{Rem}

\section{The proof of the main Theorem}\label{sec-proofofMT}

We have collected all the ingredients for the proof of our main
Theorem \ref{theorem-MT}. We start with claim (1). Let us state it again:
\begin{theorem} Let $\Lambda\subset \hfhd$ be a critical equivalence class and  suppose that $\nu\in\Lambda$ is anti-dominant. Then for all $w\in \CW(\Lambda)$ and $n\ge 0$ we have
$$
[\Delta(w.\nu):L(\nu-n\delta)]=p(n).
$$
\end{theorem}

\begin{proof} By Theorem \ref{Th:main-of-Ara}, (1), the functor $F$
  is exact. Hence we have
\begin{align*}
 \cha_q F(\Delta(w. \nu)) &=\sum_{\gamma\in\Lambda}  [\Delta(w .\nu):
 L(\gamma)]
 \cha_q F(L(\gamma))\\
&=\sum_{\gamma\in\Lambda\atop \text{$\gamma$ anti-dominant}} [\Delta(w .\nu):
 L(\gamma)]
 \cha_q F(L(\gamma))
\end{align*}
Note that $\gamma\in\Lambda$ is anti-dominant if and only if $\gamma=\nu+r\delta$ for some $r\in\DZ$. Since $\bra w. \nu,D\ket=\bra \nu,D\ket$
for all $w\in \CW$ (as $\langle\alpha,D\rangle=0$ for all finite roots $\alpha$), the claim  now follows directly  from Theorem \ref{Th:main-of-Ara}, (2).
\end{proof}

It remains to prove part (2) of Theorem \ref{theorem-MT}. Let us
recall the statement:

\begin{theorem} \label{theorem-dommul} Let $\Lambda$ be a subgeneric critical equivalence class and let $\lambda\in \Lambda$ be dominant. Denote by $\alpha$ the positive finite root with $\{\ol\lambda,s_\alpha.\ol\lambda\}$.  Then we have,  for all $n\ge 0$,
$$
[\Delta(\lambda):L(\lambda-n\delta)]=[\Delta(\alpha\uparrow\lambda):L(\lambda-n\delta)]=p(n).
$$
\end{theorem}

For the proof of the above statement we need the following result.
\begin{lemma}\label{lemma-prim} Suppose that $\lambda$, $\Lambda$ and $\alpha$ are as in Theorem \ref{theorem-dommul}. For any $n\ge 0$ we then have
$$
\dim\Delta(\lambda)_{\lambda-n\delta}^{\hfn_+}=\dim \Delta(\alpha\uparrow\lambda)_{\lambda-n\delta}^{\hfn_+} = p(n).
$$
\end{lemma}
\begin{proof} Note that $\dim\Delta(\lambda)_{\lambda-n\delta}^{\hfn_+}=\dim\Hom(\Delta(\lambda-n\delta),\Delta(\lambda))$, so Theorem \ref{theorem-ResVerma} yields $\dim\Delta(\lambda)_{\lambda-n\delta}^{\hfn_+}=p(n)$.

Let $\gamma\in\Lambda$ be arbitrary. The Jantzen sum formula yields $[\Delta(\gamma):L(\alpha\downarrow\gamma)]=1$ (note that $\alpha\downarrow\gamma$ is maximal in the set $\{\alpha\downarrow\gamma+n\delta\mid [\Delta(\gamma):L(\alpha\downarrow\gamma+n\delta]\ne 0\}$). Hence $\dim\Hom(\Delta(\alpha\downarrow\gamma),\Delta(\gamma))=1$. Since any non-trivial homomorphism between Verma modules is injective, we can view $\Delta(\alpha\downarrow\gamma)$ as a uniquely defined submodule in $\Delta(\gamma)$. 
 In particular, we have a chain of inclusions
 $$
 \Delta(\lambda)\subset\Delta(\alpha\uparrow\lambda)\subset \Delta(\alpha\uparrow^2\lambda)\subset \Delta(\alpha\uparrow^3\lambda).
 $$
 From now on we view each of these modules as a submodule of all the modules appearing on its right (note that  $\dim\Hom(\Delta(\lambda), \Delta(\alpha\uparrow^2\lambda))>1$ in general). 
 By Theorem \ref{th:FF-freeness}, there is an element $z\in \CZ^-$, uniquely defined up to multiplication with a scalar in $\DC^\times$, such that $\Delta(\lambda)=z\Delta(\alpha\uparrow^2 \lambda)$. It is easy to see that this implies $\Delta(\alpha\uparrow\lambda)=z\Delta(\alpha\uparrow^3\lambda)$. 
 
We fix  non-zero vectors $v_{\alpha\uparrow^2\lambda}$ and $v_{\alpha\uparrow^3\lambda}$  of highest weight in $\Delta(\alpha\uparrow^2\lambda)$ and $\Delta(\alpha\uparrow^3\lambda)$, resp. 
We consider both as elements in the free  $\CZ^-$-module $\Delta(\alpha\uparrow^3\lambda)$. As $\CZ^-$ is a graded polynomial ring (in infinitely many variables), as $\CZ^-_0\cong \DC$ and as $v_{\alpha\uparrow^3\lambda}$ and $v_{\alpha\uparrow^2\lambda}$ are not contained in $\left(\bigoplus_{n<0}\CZ^-_n\right)\Delta(\alpha\uparrow^3\lambda)$, we can extend the set $\{v_{\alpha\uparrow^2\lambda},v_{\alpha\uparrow^3\lambda}\}$ to a $\CZ^-$-basis of $\Delta(\alpha\uparrow^3\lambda)$.  In particular, if $v\in \Delta(\alpha\uparrow\lambda)=z\Delta(\alpha\uparrow^3\lambda)$ is of the form $\tilde z v_{\alpha\uparrow^2\lambda}$ for some $\tilde z\in\CZ^-$, then $\tilde z$ is divisible by $z$ in $\CZ^-$. 

Now let  $v\in \Delta(\alpha\uparrow\lambda)^{\hfn_+}_{\lambda-n\delta}$. Then, again by Theorem \ref{th:FF-freeness}, $v=\tilde z v_{\alpha\uparrow^2\lambda}$ for some $\tilde z\in\CZ^-$. As we have observed above, $\tilde z$ is divisible by $z$ in $\CZ^-$, hence $v$ is contained in $\Delta(\lambda)$. Hence
$$
\Delta(\alpha\uparrow\lambda)^{\hfn_+}_{\lambda-n\delta}=\Delta(\lambda)^{\hfn_+}_{\lambda-n\delta},
$$
and we conclude $\dim\Delta(\alpha\uparrow\lambda)^{\hfn_+}_{\lambda-n\delta}=p(n)$ from what we have shown earlier. 
\end{proof}
\begin{proof}[Proof of Theorem \ref{theorem-dommul}] Note first that the claimed identities are equivalent to the
 identities
\begin{align*} 
[\Delta(\lambda):L(\lambda-n\delta)]&=\dim
\Delta(\lambda)^{\hfn_+}_{\lambda-n\delta},\\
[\Delta(\alpha\uparrow\lambda):L(\lambda-n\delta)]&=\dim
\Delta(\alpha\uparrow\lambda)^{\hfn_+}_{\lambda-n\delta}
\end{align*}
as the right hand sides both equal $p(n)$ by  Lemma \ref{lemma-prim}. Hence we want to prove that each simple subquotient with highest weight $\lambda-n\delta$ of $\Delta(\lambda)$ or $\Delta(\alpha\uparrow\lambda)$ corresponds to a primitive vector. As $\Delta(\alpha\uparrow\lambda)$ can be considered as a submodule of $\Delta(\alpha\uparrow^2\lambda)$ and as $\alpha\uparrow^2\lambda$ is dominant again, it is enough to prove the above claim for the simple subquotients of $\Delta(\lambda)$, i.e.~ it is enough to prove that $[\Delta(\lambda):L(\lambda-n\delta)]=\dim
\Delta(\lambda)^{\hfn_+}_{\lambda-n\delta}$.

We prove this  by induction on the number $n$. 
The case $n=0$ is easy to settle: we certainly have $[\Delta(\lambda):L(\lambda)]=1=\dim \Delta(\lambda)^{\hfn^+}_\lambda$. So let us assume that $n>0$ and that  
$$
[\Delta(\lambda):L(\lambda-l\delta)]=\dim \Delta(\lambda)^{\hfn^+}_{\lambda-l\delta}
$$
holds for all $l<n$. Let 
$$
V(\lambda):= \Delta(\lambda)/\Delta({\alpha\downarrow\lambda})
$$
be the {\em Weyl module} with highest weight $\lambda$ (again we consider $\Delta(\alpha\downarrow\lambda)$ as a uniquely define submodule in $\Delta(\lambda)$.  Now we prove the following statement:

\begin{enumerate}
\item
{\em We have $[V(\lambda):L(\lambda-n\delta)]= \dim V(\lambda)^{\hfn_+}_{\lambda-n\delta}$.} 
\end{enumerate}
Suppose that $X$ is a submodule of $V(\lambda)$ such
  that there exists a surjection $f\colon X\to L(\lambda-n\delta)$. In order
  to prove claim (1) it is enough to show that there is a map
  $g\colon \Delta(\lambda-n\delta)\to X$ such that the composition
  $f\circ g\colon \Delta(\lambda-n\delta)\to X\to L(\lambda-n\delta)$ is surjective. 

Now $L(\lambda-n\delta)$ is not a quotient of $V(\lambda)$ (since $n>0$), hence $X$
is a proper submodule of $V(\lambda)$. So its weights are strictly
smaller than $\lambda$, i.e.~ $X$ is an object in
$\CO_\Lambda^{<\lambda}$. Since $P^{<\lambda}(\lambda-n\delta)$ is a
projective cover of $L(\lambda-n\delta)$ in $\CO_\Lambda^{<\lambda}$, there
is a map $h\colon P^{<\lambda}(\lambda-n\delta)\to X$ such that the
composition $P^{<\lambda}(\lambda-n\delta)\stackrel{h}\to
X\stackrel{f}\to L(\lambda-n\delta)$ is surjective. We are going to
show that the map $h$ factors over the quotient map
$P^{<\lambda}(\lambda-n\delta)\to
P^{\varle\lambda-n\delta}(\lambda-n\delta)\cong
\Delta(\lambda-n\delta)$, so that we get an induced map
$\Delta(\lambda-n\delta)\to X$, which we can take as the map $g$ that we
wanted to construct.

By Lemma \ref{lemma-Vermasubquots}, the module $P^{<\lambda}(\lambda-n\delta)$ is an extension of its subquotients $P(\lambda-n\delta)_{[\nu]}$.  
Now $X$ is a submodule of
$V(\lambda)$, which is a module of highest weight $\lambda$. So for each $r>0$ and $z\in \CA_r$ we have
$z^{V(\lambda)}=0$, hence $z^X=0$. Using the  Propositions \ref{prop-actonproj} and
\ref{prop-actonprojsub} and the assumption on the multiplicities, we
can inductively show that each Verma subquotient of
$P^{<\lambda}(\lambda-n\delta)$ lies in the kernel of $h$ except possibly $P(\lambda-n\delta)_{[\alpha\uparrow\lambda-n\delta]}$ and $P(\lambda-n\delta)_{[\lambda-n\delta]}$. 
But since $[X:L(\alpha\uparrow \lambda-n\delta)]=0$ by part (1) of our main Theorem, also $P(\lambda-n\delta)_{[\alpha\uparrow\lambda-n\delta]}$ is contained in the kernel, so  we get an induced map
$$
P(\lambda-n\delta)^{\varle\lambda-n\delta}\cong \Delta(\lambda-n\delta)\to X,
$$
which is what we wanted to show. Hence we proved claim (1).

Secondly, we claim:

\begin{enumerate}
\setcounter{enumi}{1}
\item
 {\em The map $\Delta(\lambda)^{\hfn_+}_{\lambda-n\delta} \to V(\lambda)^{\hfn^+}_{\lambda-n\delta}$ that is induced by the canonical map $\pi\colon\Delta(\lambda)\to V(\lambda)$, is surjective.} 
\end{enumerate}
So let $v\in
V(\lambda)^{\hfn^+}_{\lambda-n\delta}$, $v\ne 0$, and denote by $X$ the
submodule of $V(\lambda)$ which is generated by $v$. Then $X$ is a
highest weight module with highest weight $\lambda-n\delta$ and there
is a short exact sequence
$$
0\to \Delta({\alpha\downarrow\lambda})\to \pi^{-1}(X)\to X\to 0.
$$
Now  the Casimir element $C$ acts on $\pi^{-1}(X)$ as a scalar, since
$\pi^{-1}(X)$ is a submodule of a Verma module. Hence we can apply
Lemma  \ref{lemma-Casimir} and we deduce that there is a submodule $Y$
of $\pi^{-1}(X)$ of highest weight $\lambda-n\delta$ that maps
surjectively to $X$. In particular, there is a preimage of $v$ in
$\pi^{-1}(X)^{\hfn^+}_{\lambda-n\delta}\subset \Delta(\lambda)^{\hfn_+}_{\lambda-n\delta}$. So we proved part (2). 

Our next step is to   prove
\begin{enumerate}
\setcounter{enumi}{2}
\item
 {\em We have $[\Delta({\alpha\downarrow\lambda}):L(\lambda-n\delta)]= \dim \Delta({\alpha\downarrow\lambda})_{\lambda-n\delta}^{\hfn_+}$.} 
\end{enumerate}
Suppose that the claim is wrong, i.e.~ suppose that $[\Delta({\alpha\downarrow\lambda}):L(\lambda-n\delta)]>\dim
  \Delta({\alpha\downarrow\lambda})^{\hfn_+}_{\lambda-n\delta}$. Let $f\colon\Delta(\alpha\downarrow\lambda)\to\Delta(\lambda)$ be a non-zero map and let us consider the composition $g\colon \Delta({\alpha\downarrow\lambda})\stackrel{f}\to \Delta(\lambda)\to\rDelta(\lambda)$. 
  Let $K$ be the kernel and $X$  the image of $g$. Then $X$ is a highest weight module of highest weight $\alpha\downarrow\lambda$, so $[X:L({\alpha\downarrow\lambda})]=1$. As $g$ factors over the map $\Delta({\alpha\downarrow\lambda})\to\rDelta({\alpha\downarrow\lambda})$, part (1) of our main Theorem implies $[X:L({\alpha\downarrow\lambda}-l\delta)]=0$ for all $l>0$. Our induction assumption implies $[\rDelta(\lambda):L(\lambda-l\delta)]=0$ for all $0<l<n$, hence $[X:L(\lambda-l\delta)]=0$ for all $l<n$. 

Now the weights of $K$ are strictly smaller than $\alpha\downarrow\lambda$.  By the induction assumption and part (1) of the main Theorem, each subquotient of $K$ of type $L(\lambda-n\delta)$ hence corresponds to a singular vector of weight $\lambda-n\delta$.  
  Together with the assumption  $[\Delta({\alpha\downarrow\lambda}):L(\lambda-n\delta)]>\dim
  \Delta({\alpha\downarrow\lambda})^{\hfn_+}_{\lambda-n\delta}$ this allows us to deduce    $[X:L(\lambda-n\delta)]>0$. 
  So  $\lambda-n\delta$ occurs as a maximal weight in the maximal submodule of $X$. 
 Hence there is a primitive vector in $X$, hence in
  $\rDelta(\lambda)$, of weight $\lambda-n\delta$. But this contradicts
  Theorem \ref{theorem-ResVerma}. Hence our assumption $[\Delta({\alpha\downarrow\lambda}):L(\lambda-n\delta)]>\dim
  \Delta({\alpha\downarrow\lambda})^{\hfn_+}_{\lambda-n\delta}$ leads to a contradiction, so we proved claim (3).
  
  Now we can finish the proof of the theorem. Recall that we have to show that $[\Delta(\lambda):L(\lambda-n\delta)]=\dim \Delta(\lambda)^{\hfn_+}_{\lambda-n\delta}$ under the assumption that $[\Delta(\lambda):L(\lambda-l\delta)]=\dim \Delta(\lambda)^{\hfn_+}_{\lambda-l\delta}$ for any $l<n$.
  We have
 \begin{align*}
 \dim\Delta(\lambda)_{\lambda-n\delta}^{\hfn_+} &= \dim\Delta({\alpha\downarrow\lambda})_{\lambda-n\delta}^{\hfn_+}+  \dim V(\lambda)_{\lambda-n\delta}^{\hfn_+}\quad\text{ (by $(2)$)}\\
 &= [\Delta({\alpha\downarrow\lambda}):L(\lambda-n\delta)]+[V(\lambda):L(\lambda-n\delta)] \quad\text{ (by $(1)$ and $(3)$)}\\
 &= [\Delta(\lambda):L(\lambda-n\delta)], 
\end{align*}
which is what we wanted to prove.
\end{proof}

\end{document}